\documentclass[12pt]{article}
\usepackage{mlmodern}

\usepackage{pdfrender,xcolor}
\usepackage{graphicx,color}
\usepackage{amsmath,amsfonts,amssymb,amsthm,MnSymbol}
\usepackage{bbm}
\usepackage[hidelinks,colorlinks=true,linkcolor=blue,citecolor=blue]{hyperref}
\usepackage{enumerate}
\usepackage{algorithm}
\usepackage{algpseudocode}
\usepackage{soul}
\usepackage{rsfso}
\usepackage{refcheck}
\usepackage{multirow}

\usepackage{booktabs}
\usepackage{multirow}
\usepackage{graphicx}
\usepackage{adjustbox}
\usepackage{subcaption}
\usepackage{caption}

\usepackage{natbib}
\newcommand{\E}{\mathbb{E}}
\newcommand{\R}{\mathbb{R}}
\newcommand{\Prob}{\mathbb{P}}
\newcommand{\Z}{\mathbb{Z}}

\providecommand{\abs}[1]{\lvert#1\rvert}
\providecommand{\Babs}[1]{\Big\lvert#1\Big\rvert}
\providecommand{\babs}[1]{\,\big\lvert#1\big\rvert}

\newcommand{\eq}[1]{(\ref{eq:#1})}

\newcommand{\thm}[1]{Theorem~\ref{thm:#1}}

\newcommand{\pro}[1]{Proposition~\ref{pro:#1}}

\newcommand{\ass}[1]{Assumption~\ref{ass:#1}}

\newcommand{\sectn}[1]{Section~\ref{sec:#1}}

\newcommand{\pend}{\hfill \thicklines \framebox(6.6,6.6)[l]{}}

\newenvironment{proof*}[1]{\noindent {\sc  #1} \rm}{\pend}

\newtheorem{theorem}{Theorem}[section]
\newtheorem{lemma}{Lemma}[section]
\newtheorem{assumption}{Assumption}[section]
\newtheorem{proposition}{Proposition}[section]

\newcommand{\setsection}[2] {
	\setcounter{section}{#1}
	\setcounter{subsection}{0}
	\setcounter{equation}{0}
	\setcounter{conjecture}{0}
	\setcounter{assumption}{0}
	\setcounter{question}{0}
	\setcounter{definition}{0}
	\setcounter{theorem}{0}
	\setcounter{corollary}{0}
	\setcounter{lemma}{0}
	\setcounter{proposition}{0}
	\setcounter{remark}{0}
	\setcounter{appen}{0}
	\setsection*{\large \bf \thesection. #2}}

\begin{document}
	\title{\bf \Large Asymptotic product-form steady-state for
		generalized Jackson networks in multi-scale heavy traffic}
	
	\author{J.G. Dai\\Cornell University\\ \and Peter Glynn\\ Stanford University \\  \and Yaosheng Xu \\ University of Chicago}
	\date{\today}
	
	\maketitle
	\begin{abstract}
		We prove that under a multi-scale heavy traffic
		condition, the stationary distribution of the scaled queue
		length vector process in any generalized Jackson network has a product-form
		limit.  Each component in the product form follows an exponential
		distribution, corresponding to the Brownian approximation of a
		single station queue. The ``single station'' can be constructed
		precisely and its parameters have a good intuitive interpretation. 
	\end{abstract}
	
	\begin{quotation}
		
		\medskip
		
	\end{quotation}

	\section{Introduction}
	\label{sec:intro}

	\cite{Jack1957} introduced a class of first-come-first-serve (FCFS)
	queueing networks that are known today as open Jackson networks. The
	defining characteristics of a Jackson network are (a) all customers
	visiting a service station are homogeneous in terms of the service
	time distribution and the routing probabilities, and (b) all
	interarrival and service time distributions, referred to as the primitive distributions, are exponential. For a
	Jackson network with $J$ stations, the queue length vector process is a continuous
	time Markov chain (CTMC).  Jackson’s pioneering contribution is that
	when the traffic intensity at each station is less than one, the CTMC
	has a unique stationary distribution that is of product form, meaning
	that the steady-state queue lengths at the various stations in the
	network are independent.  This product-form result makes the
	computation of steady-state performance measures highly tractable.
	
	When the primitive distributions are general, the corresponding
	network is known as a generalized Jackson network (GJN). For a GJN,
	the product-form result no longer holds in
	general. 
	Creating and justifying approximations for GJNs have been an active research area for more than 60 years.
	In this paper, we prove the following asymptotic
		product-form result for a GJN:
		for a GJN with $J$ stations, satisfying
		the   multi-scale heavy traffic condition
		\begin{align} \label{eq:mscale}
		1\ll\frac{1}{1-\rho_1}\ll \frac{1}{1-\rho_2}\ll\cdots\ll\frac{1}{1-\rho_J},
		\end{align}
		the following approximation holds
		\begin{align}\label{eq:GJN}
		\Big( (1-\rho_1) Z_1, \ldots,
		(1-\rho_J) Z_J \Big)\approx \left(\rho_1d_1E_1 ,\ldots, \rho_Jd_JE_J\right),
		\end{align}
		where $E_1,\ldots, E_J$ are i.i.d. exponential random variables with mean $1$, and $(d_1, \ldots, d_J )>0$ is some deterministic vector. In (\ref{eq:GJN}), $\rho_j<1$ is the traffic intensity at station
		$j$ defined in (\ref{eq:rho}), $Z=(Z_1, \ldots, Z_J)$
		is the vector of steady-state queue length.
		The right side
		of (\ref{eq:GJN}) indicates a product-form distribution. Regarding the multi-scale heavy traffic condition (\ref{eq:mscale}), $1/(1-\rho_j)$ will be called
		the heavy traffic normalization factor for station $j$ and
		$a\ll b$ is interpreted as ``$b$ being much larger than $a$''.
		Condition (\ref{eq:mscale}) says that the
		heavy traffic normalization factors for different stations are
		all large, but \emph{of widely separated magnitudes}.  A precise
		definition of the multi-scale heavy traffic condition will be
		given in (\ref{htpar}) of Section~\ref{sec:main}, and a precise version
		of (\ref{eq:GJN}) in terms of convergence in distribution
		will be stated in Theorem~\ref{thm:main}.

	 Our result is the first of its kind. It has many potential
		ramifications.  An obvious benefit of our result is that the
		right side of (\ref{eq:GJN}) offers a scalable method to
		approximately evaluate the steady-state performance of a GJN. In
		Section~\ref{sec:num}, we demonstrate that approximations based
		on (\ref{eq:GJN}) can be accurate sometimes even when
		the load-separation condition in (\ref{eq:mscale}) is violated. Our asymptotic product-form result has been
		extended to multiclass queueing networks operating under static
		buffer priority service policies in \cite{DaiHuo2024}. An active research
		direction is to explore whether the asymptotic product-form phenomenon holds for a
		much broader class of stochastic systems including stochastic
		processing networks as defined in \cite{DaiHarr2020}. 
		These exciting developments have the potential to radically speed up the
		performance analysis and optimization of real-world complex
		systems.

	When a GJN is in conventional heavy traffic, namely all the
	stations have roughly equal heavy traffic normalization factors,
	\cite{Reim1984} and \cite{John1983} prove that the scaled queue
	length vector process converges to a multi-dimensional reflecting
	Brownian motion (RBM).  In their conventional heavy traffic limit
	theorem, one common scaling is applied (based on the heavy traffic
	normalization factor for an arbitrarily chosen station) across all
	the different stations.  Their RBMs were first introduced in
	\cite{HarrReim1981}, and the stationary distributions of these RBMs
	were shown to exist in \cite{HarrWill1987} when $\rho_j<1$ for each
	station $j$. These pioneering works have propelled
	the development of Brownian models for queueing network performance
	analysis and control in the last forty years; see, for example,
	\citet{Harr1988}, \citet{HarrNguy1993}, \cite{ChenYao2001},
	\cite{KangKellyLeeWilliams2009}, \cite{AtaHarrSi2024}, and a survey paper \cite{Will2016}.
	
	The stationary distribution of a multi-dimensional RBM is typically
	not of product-form except under a ``skew-symmetry'' condition
	characterized in \cite{HarrWill1987a}; see, e.g., \cite{Harr1978}.
	Numerical algorithms have been developed for computing the
	stationary distributions of RBMs in low dimensions; see
	\cite{DaiHarr1992} and \cite{ShenChenDaiDai2002}.  It is likely that these algorithms, even when convergent, do not scale well when the number $J$ of stations gets large. Moreover, a recent simulation-based algorithm proposed in \cite{BlanChenSiGlynn2021} has a complexity that scales linearly in dimension $J$ when the RBM data satisfies (a) a uniform (in $J$) contraction condition and (b) a uniform (in $J$) stability condition. The requirement of uniformity in $J$ is a strong condition when $J$ is large. In other words, the complexity guarantee of their work may not cover many real-world high-dimensional network examples. On the other hand, the approximation based on (\ref{eq:GJN}) can be easily computed even when $J$ is (very) large. The approximation (\ref{eq:GJN}) can be either used directly to approximate the original networks or it can be used at a ``pre-conditioner" for numerical algorithms, such as that of \cite{DaiHarr1992}. Specifically, their algorithm computes a density relative to a reference measure which must be chosen well in order to guarantee convergence. The approximation (\ref{eq:GJN}) is a possible choice for that reference measure; its effectiveness will be explored elsewhere.

	Our proof of Theorem~\ref{thm:main} employs the BAR-approach recently developed in
	\cite{BravDaiMiya2017,BravDaiMiya2024} for continuous-time queueing networks with
	general primitive distributions.  The BAR-approach was developed initially
	for providing an  alternative to the limit interchange procedure in
	\cite{GamaZeev2006}, \cite{BudhLee2009}, \cite{Gurv2014a}, and
	\cite{YeYao2018}, justifying the convergence of the prelimit
	stationary distribution to the RBM stationary distribution under
	conventional heavy traffic.
	Our proof of Theorem~\ref{thm:main} relies heavily on a steady-state moment
	bound for the scaled queue length vector. This bound is proved in a
	companion paper \cite{GuanChenDai2023}.
	
	The parameter $d_j$ in (\ref{eq:GJN}) has an explicit formula
	(\ref{eq:dj}) in terms of the first two moments of the primitive
	distributions and the routing probabilities. The distribution
	of $d_j E_j$ corresponds to the stationary distribution of a
	one-dimensional RBM approximating a single station queue. The
	“single station” can be constructed precisely and has a good
	intuitive interpretation as described immediately below the
	statement of Theorem~\ref{thm:main}, consistent with the bottleneck
	analysis advanced in \cite{ChenMand1991}.  Our work contributes to
	a line of research pioneered by \cite{Whit1983} on decomposition
	approaches for the performance analysis of GJNs, and partly justifies
	the sequential bottleneck decomposition (SBD) heuristic proposed in
	\cite{DaiNguyReim1994}. For a recent survey of decomposition
	approaches including ``robust queueing approximations'',
	readers are referred  to \cite{WhittYou2022}. We leave it as
	 future work to compare the approximation (\ref{eq:GJN}) with the extensive numerical work in \cite{WhittYou2022}.

	We note that when the normalization factors $(1-\rho_j)^{-1}$
        are widely separated, this implies that in heavy traffic, the
        diffusion temporal scalings $(1-\rho_j)^{-2}$ that dictate the
        time scales governing the dynamics at each station are also
        widely separated. It is this ``time scale separation'' that
        leads to the independence across stations. See Theorem~1 and
        Proposition 2 of \cite{ChenDaiGuang2025} for precise statements
        of the process-level independence.


	\section{Generalized Jackson networks}\label{sec:gjn}
	We first define an open generalized Jackson network, following closely
	Section 2.1 of \cite{BravDaiMiya2017} both in terms of terminology
	and notation. There are $J$ single-server stations in the
	network. Each station has an infinite-size  buffer that holds jobs waiting for service. Each
	station potentially has an external job arrival process.  An arriving job is
	processed at various stations in the network, one station at a time,
	before eventually exiting the network. When a job arrives at a station
	and finds the server busy processing another job, the arriving job
	waits in the buffer until its turn for service. After being processed
	at one station, the job either is routed to another station for
	further processing or exits the network, independent of all previous
	history.  Jobs at each station, either waiting or in service, are
	homogeneous in terms of service times and routing probabilities.  Jobs
	at each station are processed following the first-come-first-serve
	(FCFS) discipline.

	Let ${\cal J}=\{1, \ldots, J\}$ be the set of stations. The following
	notational system follows \cite{BravDaiMiya2024} closely.  For each
	station $j\in\cal{J}$, we are given two nonnegative real numbers
	$\alpha_j\in \R_+$ and $\mu_j\in \R_+$, two sequences of i.i.d.\ random
	variables $T_{e,j} = \{T_{e,j}(i), i\in\mathbb{N}\}$ and
	$T_{s,j}=\{T_{s,j}(i),i\in\mathbb{N}\}$, and one sequence of i.i.d.\
	random vectors $\Phi_j=\{\Phi_j(i), i\in \mathbb{N}\}$, where $\mathbb{N}$ denotes the set of natural numbers. We allow the possibility that some stations may not have external arrivals by assuming 
		$\alpha_j=0$.  If such a station $j$ exists, all terms associated with external arrivals at station $j$ can be removed without ambiguity. To maintain clean notation, we assume each station $j\in\cal{J}$ has external arrival. We assume
	$T_{e,1}$, \ldots, $T_{e,J}$, $T_{s,1}$, \ldots, $T_{s,J}$, $\Phi_1$,
		$\ldots$, $\Phi_J$ are independent.  Following \cite{BravDaiMiya2024},
	we assume $\E[T_{e,j}(1)] = 1$, $\E[T_{s,j}(1)] = 1$ such that
	$T_{e,j}(1)$ and $T_{s,j}(1)$ are unitized. 
	 For $j\in \cal{J}$,
	we use
	$T_{e,j}(i)/\alpha_j$ to represent the interarrival time between the
	$(i-1)$th and $i$th external arriving jobs at station $j$, and
	$T_{s,j}(i)/\mu_j$ to denote the $i$th job service time at station
	$j$.  Therefore, $\alpha_j$ is interpreted as the external arrival
	rate and $\mu_j$ as the service rate at station $j$. The random vector
	$\Phi_j(1)$ takes vector $e^{(k)}$ with probability $P_{jk}$ for
	$k\in \mathcal{J}$ and takes vector $e^{(0)}$ with probability
	$P_{j0}\equiv 1-\sum_{k\in \mathcal{J}}P_{jk}$, where $e^{(k)}$ is the
	$J$-vector with component $k$ being 1 and all other components being
	$0$, and $e^{(0)}$ is the $J$-vector of zeros. When
	$\Phi_j(i)=e^{(k)}$,  the $i$th departing job from station
	$j$ goes next to station $k\in \mathcal{J}$. When
	$\Phi_j(i)=e^{(0)}$, the job exits the network.  We assume the
	$J\times J$ routing matrix $P=(P_{ij})$ is transient, which is
	equivalent to requiring that $(I-P)$ is invertible.  This assumption ensures
	that each arriving job will eventually exit the network, and thus the
	network is known as an open generalized Jackson network.
	
	Denote by $c^2_{e, j}=\operatorname{Var}\left(T_{e, j}(1)\right)$ and
	$c^2_{s, j}=\operatorname{Var}\left(T_{s, j}(1)\right)$.  It is known
	that $c^2_{e, j}$ and $c^2_{s, j}$ are the squared coefficients of
	variation for interarrival time and service time, respectively.

	\noindent\textbf{Traffic Equation.} Given the queueing network, let
	$\lambda$ be the unique solution to the traffic equation
	\begin{equation}\label{trafficeq}
	\lambda = \alpha + P' \lambda,
	\end{equation}
	where $P'$ is the transpose of $P$. Here $\lambda_j$ is referred to as
	the nominal total arrival rate to station $j$, including the external
	arrival and the arrivals from other stations. For each station
	$j\in \mathcal{J}$, define the traffic intensity at station $j$ by
	\begin{align}\label{eq:rho}
	\rho_j=\lambda_j/\mu_j.          
	\end{align}
	We assume that
	\begin{align*}
	\rho_j< 1, \quad j\in \mathcal{J}.
	\end{align*}
	
	\noindent\textbf{Markov process.} For $t\geq0$ and $j\in\cal{J}$, let
	$Z_j(t)$ be the number of jobs at station $j$, including possibly one
	being in  service. Let $R_{e,j}(t)$ be the residual time until the
	next external arrival to station $j$. Let $R_{s,j}(t)$ be the residual
	service time for the job being processed in the station $j$. If
	$Z_j(t)=0$, the residual service time is the service time of the next job
	at station $j$, meaning $R_{s,j}(t) = T_{s,j}(i)$ for an appropriate
	$i\in\mathbb{N}$. We write $Z(t), R_e(t), R_s(t)$ as the vectors of
	$Z_j(t), R_{e,j}(t)$, and $R_{s,j}(t)$,
	respectively. Define $$X(t)=(Z(t), R_e(t), R_s(t))$$ for each
	$t\geq 0$. Then $\{X(t), t\geq 0\}$ is a Markov process with respect
	to the filtration
	$ \mathbb{F}^{X} \equiv\left\{\mathcal{F}_{t}^{X} ; t \geq 0\right\}$
	defined on the state space
	$\Z_{+}^{J} \times \R_{+}^{J} \times
	\R_{+}^{J}$, where
	$\mathcal{F}_{t}^{X}=\sigma(\{X(u) ; 0 \leq u \leq t\}) $.
	

	\section{Multi-scale heavy traffic and the  main result}
	\label{sec:main}
	We consider a family of generalized Jackson networks indexed
	by $r\in(0,1)$.  We denote by $ \mu^{(r)}_j$ the service
	rate at station $j$ in the $r$th network. To keep the
		presentation clean, we assume the service rate vector $\mu^{(r)}$ is the
		only network parameter that depends on $r\in (0,1)$, and that the service time $T_{s,j}/\mu^{(r)}_j$ depends on $r$ solely through the service rate.  The arrival rate $\alpha$, the
	unitized interarrival and
	service times, and routing vectors do not depend on $r$. Consequently,
	the nominal total arrival rate $\lambda$ does not depend on $r$. 
	
	\begin{assumption}[Multi-scale heavy traffic]\label{ass:mscale}
		We assume there is a constant $b_j>0$ for each station $j\in \mathcal{J}$ and  for $r\in (0, 1)$ 
		\begin{equation}\label{htpar}
		\begin{aligned}
		& \mu^{(r)}_j  - \lambda_j = b_j r^j,\quad j\in\mathcal{J}.
		\end{aligned}
		\end{equation}
	\end{assumption}
	Condition (\ref{htpar}) is equivalent to 
	\begin{align*} 
	1- \rho^{(r)}_j =b_j r^j/\mu^{(r)}_j, \quad j\in\mathcal{J}.
	\end{align*}
	Condition  (\ref{htpar}) implies $\mu^{(r)}_j\to \lambda_j$ and
	\begin{align*}
	1/(1-\rho_1^{(r)})= \frac{\mu^{(r)}_1}{b_1}\frac{1}{r} \to \infty, \quad \frac{1/(1-\rho_{j+1}^{(r)})}{1/(1-\rho_{j}^{(r)})} = \frac{b_j\mu^{(r)}_{j+1}}{b_{j+1
		}\mu^{(r)}_j} \frac{1}{r}\to\infty, \quad j=1, \ldots, J-1.
	\end{align*}
	as $r\downarrow 0$, which is consistent with (\ref{eq:mscale}).  We
	call condition (\ref{htpar}) the multi-scale heavy traffic condition.
	Condition (\ref{htpar}) is one way to achieve ``widely separated
	traffic normalization factors''. Clearly, it is not the only way.
	We choose to use the form in condition (\ref{htpar}) to make our
	proofs clean. We leave it to future research to understand the most general form
	of load-separation under which Theorem~\ref{thm:main} still holds. In the rest of this paper, we choose $b_j=1$ for $j\in \mathcal{J}$. When  $b_j\neq 1$, all
	the proofs remain essentially the same.

	\begin{assumption}[Moment]\label{ass:moment}
		We assume that there exists a $\delta_0>0$ such that the primitive
		distributions have $J+1+\delta_0$ moments. Namely,
		\begin{align} \label{eq:moment}
		\E\big[\big(T_{e,j}(1)\big)^{J+1+\delta_0}\big] <\infty , ~
		\E\big[\big(T_{s,j}(1)\big)^{J+1+\delta_0}\big] <\infty, ~ j\in
		\mathcal{J}.
		\end{align}
              \end{assumption}
              Condition~\eq{moment} is stronger than the
                $(2+\delta_0)$-th moment condition in
                \cite{BravDaiMiya2017}. Assumption \ref{ass:moment}
              is used to derive in Section~\ref{sec:abar} the
              asymptotic basic adjoint relationship (\ref{eq:abar})
              and Taylor expansions (\ref{eq:eta-t-taylor}) and
              (\ref{eq:xi-t-taylor}). These three equations are
              critical in the proof of Theorem~\ref{thm:main}. It
                is an open problem to explore weaker moment conditions
                under which Theorem~\ref{thm:main} holds.

              The following assumption is mild.  \cite{Dai1995} proves
              that the following assumption holds when
              $\rho_j^{(r)}<1$ for each station $j$, and the
              interarrival times are ``unbounded and spread-out''.
              
	\begin{assumption}[Stability]\label{ass:stable}
		For each $r\in(0,1)$, the Markov process $\{X^{(r)}(t), t\ge 0\}$ is positive Harris recurrent and it has a unique stationary distribution.
	\end{assumption}
	Under \ass{stable},
	we use
	\begin{align*}
	X^{(r)}=\big(Z^{(r)}, R_{e}^{(r)}, R_{s}^{(r)}\big)
	\end{align*}
	to denote the steady-state random vector that has the stationary distribution. It is well known that
	\begin{align}\label{eq:idle}
	\Prob\{Z^{(r)}_j = 0\} = 1-\rho^{(r)}_j=r^j/\mu^{(r)}_j, \quad j\in \mathcal{J}.
	\end{align}
	See, for example, (A.11) and (A.13) of \cite{BravDaiMiya2017} for a proof.
	Therefore, under multi-scale heavy traffic condition (\ref{htpar}), each server
	approaches $100\%$ utilization. 
	
	The following is the main result of this paper.

	\begin{theorem}\label{thm:main}
		Assume Assumptions~\ref{ass:mscale}-\ref{ass:stable} all hold.
		Then
		\begin{align}\label{eq:joint-conv}
		\Big((1-\rho^{(r)}_1)Z_1^{(r)}, \ldots, (1-\rho^{(r)}_J)Z_J^{(r)}\Big) \Rightarrow  \left(d_1E_1  ,\ldots,  d_JE_J\right), \quad \text{ as } r\to 0,
		\end{align}
		where ``$\Rightarrow$'' stands for convergence in
		distribution.  Furthermore, $E_1, \ldots, E_J$
		are independent exponential random
		variables with mean $1$, and for each station $j\in\mathcal{J}$,  
		\begin{align}
		d_j =&\frac{\sigma^2_j}{2\lambda_j(1-w_{jj})}, \quad \text{ and }\label{eq:dj} \\
		\sigma_j^2=& \sum_{i< j} \alpha_i\bigl ( w_{ij}^2c_{e,i}^2 + w_{ij}(1-w_{ij}) \bigr ) + \alpha_jc_{e,j}^2 +  \sum_{i>j}\lambda_i \bigl ( w_{ij}^2c_{s,i}^2 + w_{ij}(1-w_{ij})\bigr )\nonumber\\
		&  + 
		\lambda_j\bigl ( c_{s,j}^2(1-w_{jj})^2 + w_{jj}(1-w_{jj})  \bigr ),
		\label{eq:sigma}
		\end{align}
		where $(w_{ij})$ is computed from the  routing matrix $P$ through a formula given at the end this section.
	\end{theorem}
	As discussed in the introduction, the pre-limit stationary
	distribution (the left side of (\ref{eq:joint-conv})) is \emph{not} of
	product form, nor is stationary
	distribution of conventional heavy traffic limit \cite{Reim1984}.
        However, the multi-scale limits $E_1, \ldots, E_J$ are independent.
        \thm{main} suggests an obvious  approximation to the stationary distribution of the queue length vector process for a generalized Jackson network with widely separated values of $(\mu_j-\lambda_j)^{-1}$, $ j\in \mathcal{J}$. Our approximation is
	\begin{align}\label{eq:approx}
	\big( Z_1(\infty), Z_2(\infty), \ldots, Z_J(\infty)\big)
	\stackrel{d}{\approx}
   \left(\frac{\rho_1}{1-\rho_1}d_1E_1,\frac{\rho_2}{1-\rho_2}d_2E_2,\ldots,
	\frac{\rho_J}{1-\rho_J}d_JE_J\right),
	\end{align}
	where $(E_1, \ldots, E_J)$ is described in \thm{main}, and
	$\stackrel{d}{\approx}$ denotes ``has approximately the same distribution as'' (and carries no rigorous meaning beyond that imparted by \thm{main}).

	We now define the $J\times J$ matrix
	$w=(w_{ij})$ used in (\ref{eq:dj}) and (\ref{eq:sigma}).
	For each $j\in \mathcal{J}$,
	\begin{align}\label{eq:ww1}
	& ( w_{1j}, \ldots, w_{j-1,j})' = (I-P_{j-1})^{-1} P_{[1:j-1], j}, \\
	&  (w_{jj}, \ldots, w_{J,j})' = P_{[j:J],j} + P_{[j:J], [1:j-1]}(I-P_{j-1})^{-1} P_{[1:j-1], j},\label{eq:ww2}
	\end{align}
	where prime denotes the transpose, and, for $A, B\subseteq \mathcal{J}$, $P_{A, B}$ is the
	submatrix of $P$ whose entries are $(P_{k\ell})$ with $k\in A$
	and $\ell\in B$ with the following conventions: for $\ell\le k$,
	\begin{align*}
	[\ell:k]=\{\ell, \ell+1, \ldots, k\}, \quad  P_{A, k}=P_{A,\{k\}},\quad   P_{\ell, B}=P_{\{\ell\},B},  \quad P_{k}=P_{[1:k],[1:k]},
	\end{align*}
	and any expression involving $P_{A, B}$ is interpreted to be zero when  either $A$ or $B$ is the  empty set.
	The expressions in (\ref{eq:ww1}) and (\ref{eq:ww2}) may
	look complicated, but $(w_{ij})$ has the following
	probabilistic interpretation.  Recall that the routing
	matrix $P$ can be embedded within a transition matrix of a
	DTMC on state space $\{0\}\cup \mathcal{J}$ with $0$ being
	the absorbing state. For each $i\in \mathcal{J}$ and
	$j\in \mathcal{J}$, $w_{ij}$ is the probability that
	starting from state $i$, the DTMC will eventually visit
	state $j$ while avoiding visiting states $\{0, j+1, \ldots, J\}$.
	See more discussion of $(w_{ij})$ in Lemma~\ref{lem:w} in Appendix~\ref{sec:lem53}.

	The variance parameter $\sigma_j^2$ in (\ref{eq:sigma})
	has the following appealing interpretation. Consider a
	renewal arrival process $\{E(t), t\ge 0\}$ with arrival
	rate $\nu$ and squared coefficient of variation $c^2$ of
	the interarrival time. Each arrival has probability $p_j$ of 
	going to station $j$ instantly, independent of all
	previous history.  Let $\{A_j(t), t\ge 0\}$ be the
	corresponding arrival process to station $j$. It is known
	that
	\begin{align}\label{eq:var}
	\lim_{t\to \infty} \frac{\text{Var}(A_j(t))}{t} =  \bigl ( c^2 p_j^2 + p_j(1-p_j) \bigr ) \nu.
	\end{align}
	Thus, each of the first $J$ terms in the right side of
	(\ref{eq:sigma}) corresponds to the variance parameter in
	(\ref{eq:var}) of $J$ arrival sources to station $j$.  The
	first $j-1$ terms correspond to arrival sources from
	lightly-loaded stations, the $j$th term corresponds to the
	external arrival process to station $j$, and the last $J-j$ terms
	correspond to arrival sources from service completions at
	heavily-loaded stations. For a lightly-loaded station
	$\ell<j$, its external arrival process serves as the
	renewal arrival process, each arrival having  probability $w_{\ell j}$ to join station $j$.  For a heavily-loaded station
	$k>j$, its service process serves as the renewal arrival
	process (as if the server is always busy),
	each arrival having  probability $w_{k j}$ to join station $j$. The resulting
	formula $d_j$ in (\ref{eq:dj}) is consistent with the
	diffusion approximation of a $\sum_i G_i/GI/1$ queue with
	$J$ external arrival processes and immediate feedback
	probability $w_{jj}$ after the service at station $j$.
	
     \thm{main} will be proved in \sectn{outline}. The proof relies on the BAR in Section \ref{sec:bar} and the asymptotic BAR in \sectn{abar}. In addition, the proof relies critically on a
	uniform moment bound that is proved in the companion paper
	\cite[Theorem 2]{GuanChenDai2023}. For easy reference, we state that result here.
	\begin{proposition}[Theorem 2 of \cite{GuanChenDai2023}]
		\label{pro:zbound}
		Assume  Assumptions~\ref{ass:mscale}--\ref{ass:stable}. There exists $r_0\in (0,1)$ and $\epsilon_0>0$
		such that
		\begin{equation}\label{eq:Jmoment}
		\sup_{r\in (0, r_0)}\E\Big[\big(r^jZ_j^{(r)}\big)^{J+\epsilon_0}\Big]<\infty ,\quad j\in\mathcal{J}.
		\end{equation}
	\end{proposition}

	The moment bound (\ref{eq:Jmoment}) holds of course when  interarrival
	and service time distributions are exponential because $Z_j^{(r)}$ is geometrically distributed with mean $\rho^{(r)}_j/(1-\rho^{(r)}_j)$. In the exponential case, Lemma~23
	of \cite{Xu2022} provides an alternative proof via Lyapunov functions.

	\section{Numerical study}
	\label{sec:num}
	In this section, we illustrate the value of the approximation
	(\ref{eq:approx}) that is rooted in Theorem~\ref{thm:main}.  In
	Section \ref{sec:formula}, we propose an approximation
	formula for the cumulative distribution function (cdf) of the queue length
	vector. In Section \ref{sec:simulation}, we conduct a simulation study on a
	two-station queueing network to assess the accuracy of the approximation.
	We experiment with three sets of network parameters. \footnote{Simulation codes are publicly available at  
		\url{https://github.com/yaoshengxu/multi-scale-GJN}.	} In the first set, we vary the variability of the primitive distributions; in the second and third sets, we fix the distributions but vary the levels of load separation and the traffic intensities, respectively. In all cases, our proposed approximation performs well across multiple performance metrics, except when there is no load separation.

	\subsection{Approximation formula}\label{sec:formula}
	We aim to approximate the steady-state queue length vector
	$Z(\infty) \equiv (Z_1(\infty), \ldots, Z_J(\infty))$ by the vector
	of exponential random variables $E\equiv (E_1, \ldots, E_J)$ as specified in (\ref{eq:approx}).  To approximate the probability mass function of $Z(\infty)$, it is
	important to note that $Z(\infty)$ on the left side of
	(\ref{eq:approx}) takes discrete values, whereas the exponential
	random variables $E$ on the right-hand side are continuous.
	We propose to approximate $Z(\infty)$ by
	$Y\equiv(Y_1, \ldots, Y_J)$, where $Y_1, \ldots, Y_J$
	are independent and each $Y_j$ has a mixed distribution given below. From the
	distribution of $Y$, we further develop  approximations that are
	tailored to the particular performance measures of interest.
	
	We define $Y_j$ to be the random variable 
	with the following cumulative distribution function (cdf):
	\begin{equation}\label{eq:hybrid}
	F_j(x) = (1-\rho_j) + \rho_j (1-e^{-\frac{x}{\rho_jd_j/(1-\rho_j)}}),\text{ for all } x\ge 0 \text{ and } F_j(0-)=0.
	\end{equation}
	In this approximation (\ref{eq:hybrid}), $Y_j$ has a positive mass at $0$, namely, 
	\begin{align*}
	\Prob\{Y_j=0\} = 1-\rho_j,
	\end{align*}
	which is equal to $\Prob\{Z_j(\infty)=0\}$ because of (\ref{eq:idle}), and  $Y_j$ is exponentially distributed with mean $\rho_jd_j/(1-\rho_j)$ on $(0,\infty)$ with total mass $\rho_j$.
	
	\textbf{Steady-state mean approximation.}
	For the mean steady-state queue length $Z_j(\infty)$,  we propose
	\begin{equation}\label{eq:mean}
	\E[Z_j(\infty)] =\frac{\rho_j}{1-\rho_j}d_j; 
	\end{equation}
	which is consistent with (\ref{eq:approx}).
	
	\textbf{Probability mass function approximation.}
	For the probability mass approximation of $Z_j(\infty)$, we directly assign $1-\rho_j$ as a fixed probability for approximating $Z_j(\infty)=0$. Then we split the remaining probability mass $\rho_j$ for $Z_j(\infty)$ by discretizing $Y_j$, such that the approximated probabilities of $Z_j(\infty)$ sum up to 1. Specifically, with this mixed discretization, $Z_j(\infty)$ has a probability mass function (pmf):
	\begin{equation}\label{eq:hybrid_z}
	\begin{aligned}
	\Prob (Z_j(\infty)=0)&  \approx \Prob(Y_j=0)=  F_j(0) = 1-\rho_j,\\
	\Prob(Z_j(\infty)=k)& \approx \Prob(k<Y_j\le k+1)=  F_j(k+1)-F_j(k)\\
	&= \rho_j\big(e^{-\frac{k(1-\rho_j)}{\rho_j d_j}} - e^{-\frac{(k+1)(1-\rho_j)}{\rho_j d_j}}\big),\quad k \in\mathbb{N}.
	\end{aligned}
	\end{equation}
	We will use the pmf specified in~(\ref{eq:hybrid_z}) to plot the histogram as shown in Section \ref{sec:simulation}.
	
	\subsection{Simulation}\label{sec:simulation}
	\begin{figure}
		\centering
		\includegraphics[width=0.7\linewidth]{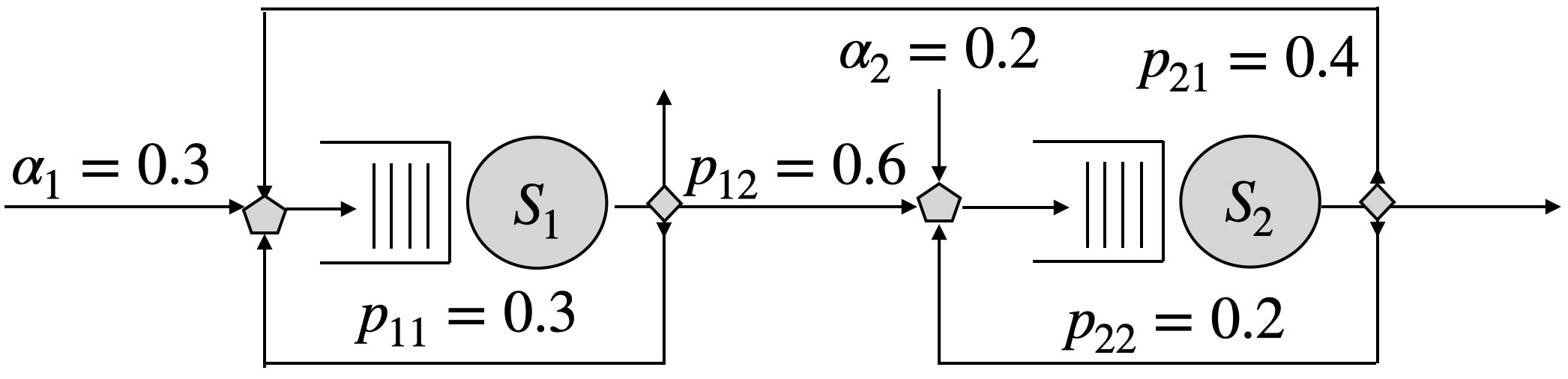}
		\caption{Two-station generalized Jackson network}
		\label{fig:gjn}
	\end{figure}
	Consider an example of two-station generalized Jackson network shown in Figure \ref{fig:gjn}
	with external interarrival and service time following Gamma
		distributions. The external arrival rates and the routing matrix are fixed to be
	$$\alpha_1=0.3,\quad \alpha_2=0.2,\quad P=\begin{pmatrix}0.3&0.6\\0.4&0.2\end{pmatrix},$$
	from which one has the  nominal arrival rates  $\lambda_1=1$ and  $\lambda_2=1$.
	We set the  service rates to be
	\begin{align*}
	\quad \mu_1=\lambda_1/\rho_1, \quad \mu_2=\lambda_2/\rho_2,
	\end{align*}
	with the traffic intensity $\rho_1, \rho_2$ being the parameters.
	
	Recall that a Gamma distribution is characterized by two
        parameters: shape and scale. When the mean of the Gamma
        distribution is fixed as the reciprocals of the external
        arrival and service rates, the distribution can be fully
        specified by choosing the shape parameter, with the scale
        parameter determined as the mean divided by the
        shape. Accordingly, the second and third columns of Table
        \ref{table:mean} represent our choices for the shape
        parameters of the external arrival and service
        distributions. Each entry under the shape columns contains two
        values, representing the shape parameters for stations 1 and
        2, respectively. Column $d$ in Table \ref{table:mean},
        representing $d_1$ and $d_2$ for station 1 and 2, is
          calculated from formula (\ref{eq:d_sim}) below that
          specializes (\ref{eq:dj})
          to the current setting:
	\begin{equation}\label{eq:d_sim}
	\begin{aligned}
	d_1&=\frac{1}{2\lambda_1(1-w_{11})}\left\{\alpha_1c_{e,1}^2 + \lambda_2 \bigl ( w_{21}^2c_{s,2}^2 + w_{21}(1-w_{21})\bigr)
	+ \lambda_1\bigl ( c_{s,1}^2(1-w_{11})^2 + w_{11}(1-w_{11})  \bigr )\right\},\\
	d_2&=\frac{1}{2\lambda_2(1-w_{22})} \left\{\alpha_1\bigl ( w_{12}^2c_{e,1}^2 + w_{12}(1-w_{12}) \bigr ) + \alpha_2c_{e,2}^2 
	+ \lambda_2\bigl ( c_{s,2}^2(1-w_{22})^2 + w_{22}(1-w_{22})  \bigr )\right\},
	\end{aligned}
	\end{equation}
	where the matrix $w$
	is calculated from (\ref{eq:ww1}) and (\ref{eq:ww2}) and its value is
	\begin{align*}
	w_{11} = 0.3,\quad  w_{12} = 0.857, \quad w_{21} = 0.4, \quad w_{22}=0.543.
	\end{align*}
	We conduct three sets of simulation experiments, summarized in {Tables~\ref{table:mean}--\ref{table:intensity}}. In the first set (Table~\ref{table:mean}), we analyze three cases---labeled A, B, C---by varying the variability of the primitive distributions. In the second set (Table~\ref{table:separation}), we examine three cases---labeled D, E, F---by varying the degree of load separation. In the third set (Table~\ref{table:intensity}), we consider three cases---labeled G, H, I---by varying the levels of traffic intensities. In Table~\ref{table:mean}, we fix the traffic intensities at $\rho_1 = 92\%$ and $\rho_2 = 98\%$, and simulate three cases by adjusting the shape parameters of the arrival and service distributions. In Tables~\ref{table:separation} and \ref{table:intensity}, we adopt the shape parameters of the Gamma distribution from Case C. Table~\ref{table:separation} simulates three cases by varying the multi-scale separation between the two stations. This is achieved by fixing $\rho_2 = 98\%$, and adjusting the traffic intensity $\rho_1$ from no separation (98\%) to large separation (88\%). Note that we could (but did not) add Case C  to  Table~\ref{table:separation} because the load in Case C lies between that of Case E and Case F. Table~\ref{table:intensity} simulates three cases by varying the levels of traffic intensities while fixing the load separation. Specifically, we fix the load separation by parameterizing $\rho_1 = 1 - r$ and $\rho_2 = 1 - r^2$, and vary $r$ from 0.2 to 0.05. The corresponding traffic intensities are presented in Table~\ref{table:intensity}. Note that the notation $r$ is slightly abused here to informally characterize the degree of separation; it does not exactly match the definition used in (\ref{htpar}). However, all results in this section are derived directly in terms of $\rho$, so this slight abuse of notation should not cause confusion. 
	
        Throughout this numerical study, we compare the simulation estimates (denoted as Sim in tables) with the theoretical estimates (denoted as M-Scale in tables) derived in Section \ref{sec:formula}. The simulation estimates are computed as follows. \cite{Whitt1989} indicates, in order to achieve a reasonable statistical precision, the run time needs to be longer if the traffic intensity is higher. In Tables~\ref{table:mean} and \ref{table:separation}, each simulation case is run for $10^9$ time units, by which time the network appears, from our experiments, to have approximately reached steady state.
        For Table~\ref{table:intensity}, considering their different levels of traffic intensities, we run Cases G, H, I for $10^7$, $10^8$, $10^9$ time units, respectively. In all sets of simulations, we record the queue lengths during the final one-tenth of the simulation period. This segment is divided into 20 batches to construct a 95\% confidence interval for the steady-state mean of $Z(\infty)$ using the batch means method. Within each batch, we compute the time average of the queue lengths by tracking the total time the queue spends in each state. Specifically, letting $t_i$ denote the total time in a batch when the queue length is equal to $i$, the time-weighted average queue length is given by
		$$
		\frac{1}{t} \sum_{i=0}^{\infty} i \cdot t_i,
		$$
	where $t$ is the total duration of the batch. The resulting batch means with confidence intervals are reported under the Sim columns in Tables~\ref{table:mean}--\ref{table:intensity}. In
	comparison, the multi-scale means calculated by (\ref{eq:mean}) are reported under the M-Scale columns.
	
	\begin{table}[H]
	\centering
	\begin{tabular}{cccccccccc}
		&&&~& \multicolumn{5}{c}{Traffic intensity $(92\%,98\%)$} & \\
		\hline
		\multirow{2}{*}{Case} & Arrival & Service& d& \multicolumn{2}{c}{Station 1} &~& \multicolumn{2}{c}{Station 2}  \\
		\cline{5-6}\cline{8-9}   &shape & shape & (\ref{eq:d_sim}) & \text { Sim} &   \text {M-Scale}  &~& \text { Sim} &   \text {M-Scale} \\
		\hline
		A &[0.75, 0.8]& [0.95, 0.6]&[1.17,1.29] &$13.37 \pm 0.12$ & 13.41 & ~& $63.58 \pm 2.06$& 63.08 \\
		\hline
		B&[0.4, 0.5]& [0.65, 0.5]  &[1.62, 1.81] &$18.65\pm 0.13$ & 18.68 & ~& $90.92\pm 4.21$& 88.64 \\
		\hline
		C&[0.2, 0.45]& [0.5, 0.4]  &[2.38, 2.57] & $27.52\pm 0.17$ & 27.35 & ~&  $121.72\pm 6.34$ & 126.15\\
	\end{tabular}
	\caption{Estimates of mean queue length on different arrival and service distributions.}
	\label{table:mean}
\end{table}

		\begin{table}[H]
	\centering
	\begin{tabular}{cccccccc}
		\multirow{2}{*}{Case}  &{Traffic}& \multicolumn{2}{c}{Station 1} &~& \multicolumn{2}{c}{Station 2}  \\
		\cline{3-4}\cline{6-7}  &intensity & \text { Sim} &   \text {M-Scale}  &~& \text { Sim} &   \text {M-Scale}   \\
		\hline D& $(98\%,98\%)$ &$118.39\pm 5.15$ &116.55  &~& $107.25\pm 6.12$&126.15\\
		\hline
		E&$(95\%,98\%)$&$45.86\pm 1.04$ & 45.19&  ~&$118.24\pm 8.69$ & 126.15\\
		\hline 
		F& $(88\%,98\%)$& $16.99\pm0.12$ & 17.44&~&$122.94\pm 5.82$ &126.15 \\	
	\end{tabular}
\caption{Estimates of mean queue length on different levels of separation\\
	\centering {using shape parameters in Case C.}}
	\label{table:separation}
\end{table}

		\begin{table}[H]
		\centering
		\begin{tabular}{cccccccc}
			\multirow{2}{*}{Case}  & \multirow{2}{*}{$r$}  &{Traffic}& \multicolumn{2}{c}{Station 1} &~& \multicolumn{2}{c}{Station 2}  \\
			\cline{4-5}\cline{7-8}  & &intensity & \text { Sim} &   \text {M-Scale}  &~& \text { Sim} &   \text {M-Scale}   \\
			\hline G& $0.2$ & $(80\%,96\%)$ &$9.15\pm 0.41$ & 9.51 &~& $56.45\pm 8.87$&61.79 \\
			\hline
			H& $0.1$&$(90\%,99\%)$&$20.96\pm 0.44$ & 21.41&  ~&$226.99\pm 37.60$ & 254.88\\
			\hline 
			I&$0.05$& $(95\%,99.75\%)$& $44.51\pm0.67$ & 45.19&~&$1215.41\pm 372.60$ &1027.23 \\	
		\end{tabular}
		\caption{Estimates of mean queue length on different levels of traffic intensities\\
			\centering {using shape parameters in Case C.}}
		\label{table:intensity}
	\end{table}


	According to Tables \ref{table:mean}--\ref{table:intensity}, our multi-scale approximation in (\ref{eq:mean}) of the mean queue length demonstrates strong performance. In Table \ref{table:separation}, the multi-scale estimates align closely with the simulation ones for different levels of load separation, except for Case D when there is no load separation. This suggests that the presence of multi-scale separation gives accurate estimation of mean values.
	

	One may wonder how well the approximation performs beyond estimating the mean queue lengths. To evaluate this, we present a histogram comparison for Case H, shown in Figure~\ref{fig:hist}. In these histograms, the Sim and Sim CI represent the mean and 95\% confidence intervals of the percentage of time the queue length takes on different values over 20 simulation batches. In contrast, the M-Scale represents the histogram derived using the approximation in (\ref{eq:hybrid_z}). Note that the zoomed-in histograms in the third panels of stations 1 and 2 use different x- and y-axis scales. Figures~\ref{fig:hist} provides further evidence that our theoretical framework offers a highly accurate approximation, extending beyond mean estimates.
	
\begin{figure}[H]
	\centering
	
	\parbox{\textwidth}{\centering {\small Station 1}\par\vspace{0.em}}
	
	\begin{subfigure}[b]{0.32\linewidth}
		\includegraphics[width=\linewidth]{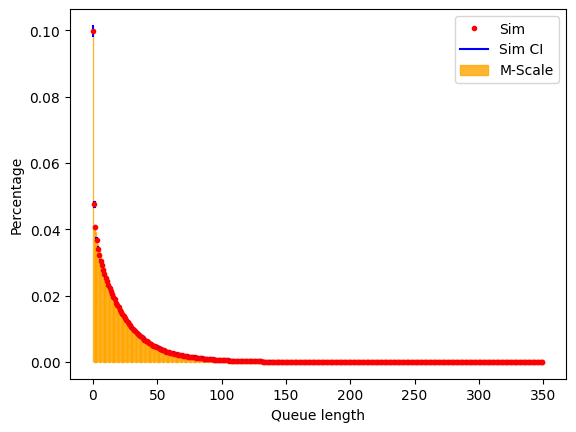}
	\end{subfigure}
	\begin{subfigure}[b]{0.32\linewidth}
		\includegraphics[width=\linewidth]{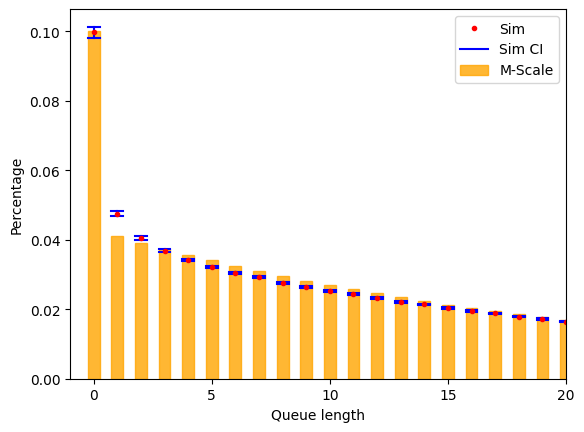}
	\end{subfigure}
	\begin{subfigure}[b]{0.32\linewidth}
		\includegraphics[width=\linewidth]{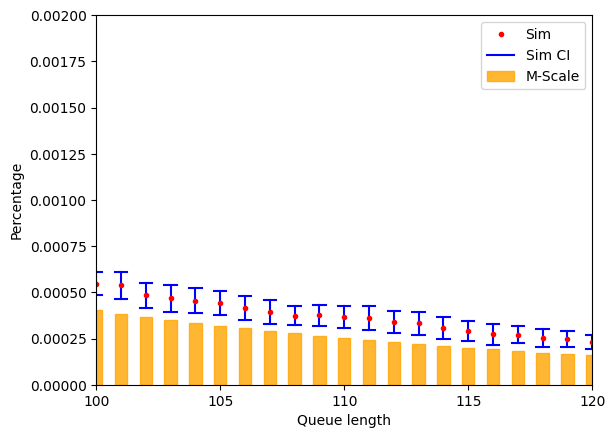}
	\end{subfigure}
	
	\parbox{\textwidth}{\centering {\small Station 2}\par\vspace{0.0em}}
	
	\begin{subfigure}[b]{0.32\linewidth}
		\includegraphics[width=\linewidth]{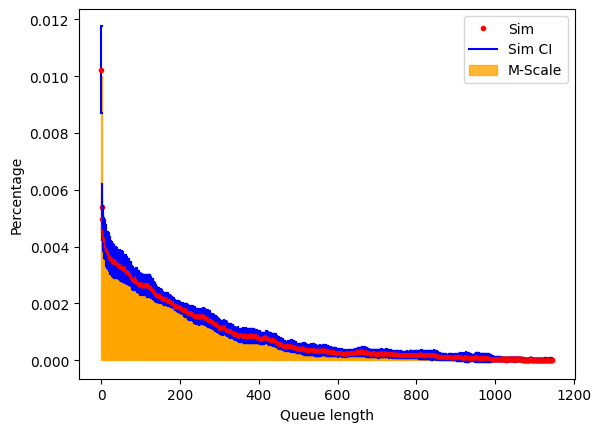}
	\end{subfigure}
	\begin{subfigure}[b]{0.32\linewidth}
		\includegraphics[width=\linewidth]{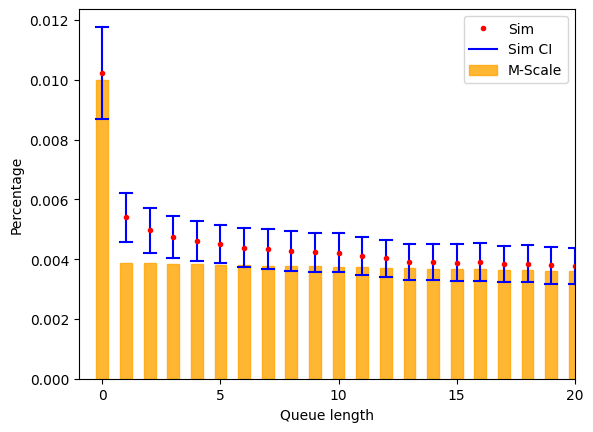}
	\end{subfigure}
	\begin{subfigure}[b]{0.32\linewidth}
		\includegraphics[width=\linewidth]{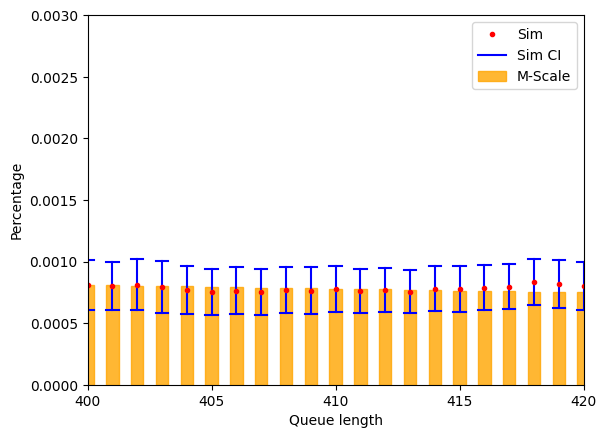}
	\end{subfigure}
	
	\caption{Case H histograms with zoom-in}
	\label{fig:hist}
\end{figure}

	
	
	Our Theorem \ref{thm:main} establishes that the multi-scale limits $E_1$ and $E_2$ are independent. To conclude this section, we take Case G as an example to investigate independence through statistical testing. To obtain independent samples, we run the simulation for a longer duration of $4\times 10^9$ time units. For the last $3 \times 10^9$ time units, we divide them into 3000 intervals, each containing $10^6$ time units. At the start of each interval, we record the joint and marginal queue lengths as representative points. These representative points, observed over $10^6$-unit intervals, are considered independent occurrences. As a result, we obtain $3000$ independent observations and count their occurrences to construct the contingency table. Using the contingency table, we perform the $\chi^2$ test of independence  \cite[Chapter 10.3]{DegSch2011}. The resulting contingency table has dimensions $70\times 278$. The $\chi^2$ statistic is $18809.59$ with $19113$ degrees of freedom. The $p$-value is $0.94$, meaning we cannot reject the null hypothesis that the queue lengths are independent. This result further supports the asymptotically independent behavior.

\section{Proof of Theorem \ref{thm:main}}\label{sec:outline}

We first present a roadmap for proving \thm{main}.
Recall that $X^{(r)}=(Z^{(r)}, R_{e}^{(r)}, R^{(r)}_s)$ denotes the
steady state of the Markov process that governs the dynamics of the
GJN.  Define the moment generating function (MGF) $\phi^{(r)}(\theta)$
of $Z^{(r)}$:
\begin{align*}
   \phi^{(r)}(\theta) = \E\Big[ \exp{\Big(\sum_{j\in \mathcal{J}}\theta_jZ_j^{(r)}\Big)}\Big], \quad  \theta\in \R^J_-\equiv \{y\in \R^J, y\le 0\}.
\end{align*}
Since  $\theta$ is restricted to be non-positive, $\phi^{(r)}(\theta)$ is
also known as the Laplace transform of $Z^{(r)}$.
One can check that  $\phi^{(r)}(r\eta_1, \ldots, r^J \eta_J)$, with
$\eta=(\eta_1,\ldots, \eta_J)'\in \R^J_-$, is the MGF of $(rZ^{(r)}_1,\ldots, r^JZ^{(r)}_J)$.
It is well known that  proving (\ref{eq:joint-conv}) in \thm{main} is equivalent to proving
\begin{align}\label{eq:phi-conv}
\lim_{r\to0}\phi^{(r)}(r\eta_1, \ldots, r^J \eta_J) = \prod_{j\in \mathcal{J}} \frac{1}{1-d_j\eta_j}, \quad   \eta=(\eta_1,\ldots, \eta_J)'\in \R^J_-.
\end{align}
To prove (\ref{eq:phi-conv}) and hence \thm{main}, it suffices to prove the following proposition because 
equation (\ref{eq:phi-conv})  follows from (\ref{eq:ite}) in
Proposition~\ref{lem:ite} by induction.
\begin{proposition}\label{lem:ite}
  Assume Assumptions~\ref{ass:mscale}-\ref{ass:stable} hold.
  Fix  $\eta\in \R^J_-$. For each $j\in \mathcal{J}$, as $r\to 0$, 
  \begin{align}
    \phi^{(r)}(0,\ldots, 0, r^j\eta_j, \ldots, r^J\eta_J) -\frac{1}{1-d_j\eta_j}
    \phi^{(r)}(0,\ldots, 0, r^{j+1}\eta_{j+1}, \ldots, r^J\eta_J) \to 0, \label{eq:ite}
  \end{align}
  with the convention that
  $\phi^{(r)}(0,\ldots, 0, r^{j+1}\eta_{j+1}, \ldots, r^J\eta_J)=1 $
  when $j=J$.
\end{proposition}

To prove Proposition~\ref{lem:ite} and hence \thm{main}, we utilize the basic adjoint
relationship (BAR) that was developed in \cite{BravDaiMiya2017, BravDaiMiya2024}. Since
the BAR is a key tool to prove our results,  we
recapitulate the BAR in Section~\ref{sec:bar}.
To make our proof strategy clear, we will build the argument in two
stages.  In Section~\ref{sec:proof-b}, we present a proof of
Proposition~\ref{lem:ite} under a simplifying assumption that the
primitive distributions have bounded supports (see
Assumption~\ref{ass:bound-support} below).  This allows us to clearly
highlight the core methodology of the BAR-approach and our key
innovations.  Then, in Section~\ref{sec:abar}, we will show how this same proof
structure extends to the general case (without assuming Assumption~\ref{ass:bound-support}) by introducing the necessary
truncation arguments and establishing an asymptotic version of the
BAR.


\subsection{Basic adjoint relationship}\label{sec:bar}
The BAR characterizes the distribution of $X^{(r)}$, and it is developed in
\cite{BravDaiMiya2017,BravDaiMiya2024} for steady-state analysis of
continuous-time queueing networks with general primitive
distributions. In this section, we recapitulate the BAR.
The BAR-approach provides an alternative to the
interchange of limits, an approach pioneered in \cite{GamaZeev2006} for proving steady-state
convergence of queueing networks in conventional heavy traffic. This
paper and the subsequent paper \cite{DaiHuo2024} demonstrate that BAR is also a natural
technical tool to prove asymptotic steady-state independence under
multi-scale heavy traffic.

To state the BAR for a GJN,  denote by $\mathcal{D}$ the set of bounded functions $f: S\to \R$ satisfying (a) for each $z\in \Z^J_+$,  $f(z, \cdot, \cdot): \R_{+}^J \times \R_{+}^J \rightarrow \R$ is continuously differentiable
at all but finitely many points, and (b) the derivatives of $f(z, \cdot , \cdot)$ in each dimension have a uniform bound over $z\in\Z^J_+$. The BAR takes the form: for each $f\in \mathcal{D}$,
\begin{align} 
  \E_\pi[\mathcal{A} f(X)] & + \sum_{j \in \mathcal{J}} \alpha_j \E_{e, j}\left[f(X_+)-f(X_-)\right] 
+\sum_{j \in \mathcal{J}} \lambda_j \E_{s, j}\left[f(X_+)-f(X_-)\right]=0,
\label{eq:fullbar}                                     
\end{align}
where
\begin{align}\label{eq:a}\mathcal{A} f(x)=-\sum_{j \in \mathcal{J}} \frac{\partial f}{\partial r_{e, j}}(x)-\sum_{j \in \mathcal{J}} \frac{\partial f}{\partial r_{s, j}}(x) \mathbbm{1}\left(z_j>0\right), \quad x=\left(z, r_e, r_s\right) \in \Z^J_+\times \R^J_+\times \R^J_+. \end{align}
In (\ref{eq:fullbar}), $\E_{\pi}[\cdot ]$ is the expectation when $X$
follows the stationary distribution $\pi$, $\E_{e,i}[\cdot]$ is the
expectation when $X_-$ follows the Palm distribution $\nu_{e,i}$. The
Palm distribution $\nu_{e,i}$ measures the distribution of the state
$X_-=(Z_-, R_{e,-}, R_{s,-})$, immediately before an external arrival
to station $i$. In particular, it is necessarily true that under the
Palm distribution $\nu_{e,i}$, $R_{e, i, -}=0$ and the post-jump state $X_+=(Z_+, R_{e, +}, R_{s,+})$ satisfying
\begin{align}\label{eq:palm-change}
  Z_{i,+}=1 + Z_{i,-}, \quad  R_{e,i,+}= T_{e,i}/\alpha_i,
\end{align}
and other components of $X_+$ remain the same as $X_-$, where
$T_{e,i}$ is a  random variable that has the same distribution
as  $T_{e,i}(1)$ and is independent of $X_-$. Equation
(\ref{eq:palm-change}) states that an external arrival to station
$i$ causes queue length to increases by $1$ and the remaining
interarrival time is reset to a ``fresh'' interarrival time. In vector form, under Palm measure $\nu_{e,i}$,
\begin{align*}
 X_+-X_-= \Delta_{e, i} \equiv\left(e^{(i)}, e^{(i)} T_{e, i} / \alpha_i, 0\right), \quad i\in \mathcal{J},
\end{align*}
where $e^{(i)}$ is the $J$-vector with the $i$th component being $1$ and other components being $0$. Palm expectation $\E_{s,i}[\cdot]$ with respect to
the Palm distribution $v_{s,i}$ can be explained similarly with
\begin{align*}
  \Delta_{s,i} \equiv\left(-e^{(i)}+\phi_i, 0, e^{(i)} T_{s, i} / \mu_i\right), \quad i \in \mathcal{J} ,
\end{align*}
and $(\phi_i, T_{s,i})$ is the random vector that has the same
distribution as $(\Phi_i(1), T_{s,i}(1))$ and is independent of $X_-$.
The proof of (\ref{eq:fullbar}) is given in (6.16) and Lemma~6.3 of
\cite{BravDaiMiya2024}, and the Palm distributions are defined in (6.11)
and (6.21) there. The proof
follows from the stationary equation
\begin{align*}
  \E[f(X(1))-f(X(0))]=0, \quad \forall f \in \mathcal{D},
\end{align*}
and the fundamental theorem of calculus, where $\{X(t), t\geq 0\}$ is
the Markov process describing the GJN dynamics and $X(0)$ follows the stationary distribution $\pi$.
BAR (\ref{eq:fullbar}) is the same as BAR (3.15) of
\cite{BravDaiMiya2017}, except that  Palm
distributions were not defined explicitly in \cite{BravDaiMiya2017}.

When $f\in \mathcal{D}$ satisfies
\begin{align}
  \label{eq:palm-kille}
  &  \E[f(z+e^{(i)}, u+  e^{(i)}T_{e,i}/\alpha_i, v)] = f(z, u, v)
\end{align}
for each  $x=(z,u, v)\in S$  with $ u_i=0$, and 
\begin{align}
  \label{eq:palm-kills}
&  \E[f(z-e^{(i)}+\phi_i, u, v+e^{(i)}T_{s,i}/\mu^{(r)}_i)] = f(z, u, v)
\end{align}
for each  $x=(z,u, v)\in S$  with  $z_i>0$ and $v_i=0$,
the Palm terms in (\ref{eq:fullbar}) are all zero and  BAR (\ref{eq:fullbar}) reduces to
\begin{align}
  \label{eq:bar-no-palm}
    \E[\mathcal{A}f(X)]=0.
\end{align}
Equation (\ref{eq:bar-no-palm}) was first proved in  Lemma~3.1 of
\cite{BravDaiMiya2017}.

\subsection{Proof of Proposition~\ref{lem:ite} under Assumption \ref{ass:bound-support}}\label{sec:proof-b}
\begin{assumption}[Bounded support]\label{ass:bound-support}
  For each $i\in \mathcal{J}$, random variables $T_{e,i}(1)$ and
  $T_{s,i}(1)$ have bounded supports. In addition, they satisfy
  \begin{align}
    \label{eq:Tp}
    \Prob\{T_{e,i}(1)>0\}=1, \quad    \Prob\{T_{s,i}(1)>0\}=1, \quad i\in \mathcal{J}.
  \end{align}
\end{assumption}
Assumption~\ref{ass:bound-support} is stronger than the moment
assumption, Assumption~\ref{ass:moment}.  The rest of this section is
devoted to the proof of Proposition~\ref{lem:ite} under
Assumption~\ref{ass:bound-support}.   In
Section~\ref{sec:abar}, we will present the second  proof of
Proposition~\ref{lem:ite} without assuming Assumption~\ref{ass:bound-support} by going through a detailed truncation argument.


Instead of working with
$\phi^{(r)}(\theta)$ directly, we will work with
\begin{align}\label{eq:psi0}
  {\psi}^{(r)}(\theta) =\E[{f}_\theta(X^{(r)})],
  \quad \text{ and } \quad    {\psi}^{(r)}_i(\theta)=\E[  f_\theta(X^{(r)})
  \mid Z^{(r)}_i=0], \quad  i\in \mathcal{J},
\end{align}
for each $\theta\in \R^J_-$, where for  $x=(z, u, v)\in \Z^J_+\times \R^J_+\times \R^J_+$,
      	\begin{align}\label{eq:f_tilde}
	{f}_{\theta}(x) = \exp\Big(\langle \theta, z\rangle - \sum_{i} \alpha_i \gamma_i (\theta_i) u_i - \sum_{i} \mu^{(r)}_i \xi_i(\theta) v_i\Big).
	\end{align}
In (\ref{eq:f_tilde}), $\gamma_i(\theta)$ and $\xi_i(\theta)$ are defined via 
	\begin{align}
	& e^{\theta_i} \E[ e^{-\gamma_i(\theta_i)T_{e,i}(1)}]=1, \label{eq:eta-g}\\
	& \bigl ( P_{i0} e^{-\theta_i} + \sum_{j\in \mathcal{J}} P_{ij} e^{\theta_j-\theta_i} \bigr )\E[ e^{-\xi_i(\theta)T_{s,i}(1)}]=1. \label{eq:xi-g}
	\end{align}
        Fix a $\theta\in \R^J_-$. When $(u,v)$ is restricted to a
        bounded domain, the function $f_\theta(z, u, v)$ belongs to
        ${\cal D}$, the class of functions defined in the second
        paragraph of Section~\ref{sec:bar}, for any choice of
        $\gamma_i(\theta)$ and $\xi_i(\theta)$ that do not necessarily
        satisfy \eq{eta-g} and \eq{xi-g} because they are independent of state variable  $x=(z,u, v)$.  Under
        Assumption~\ref{ass:bound-support}, $\gamma_i(\theta_i)$ and
        $\xi_i(\theta)$ in \eq{eta-g} and \eq{xi-g} are well defined
        for \emph{all} $\theta\in \R^J$, and one can check that when
        $\gamma_i(\theta)$ satisfies \eq{eta-g}, $f_\theta$ satisfies
        \eq{palm-kille}, and when $\xi_i(\theta)$ satisfies \eq{xi-g},
        $f_\theta$ satisfies \eq{palm-kills}.

        
        The advantage of working with ${\psi}^{(r)}$ is that
  it, together with 
        $({\psi}^{(r)}_1, \ldots,  {\psi}^{(r)}_J)$, 
         satisfies the following transform version of BAR.
The transform-BAR
(\ref{eq:bar-b}) follows from BAR (3.20) of \cite{BravDaiMiya2017} by
taking $f(x)=f_\theta(x)$ in (\ref{eq:f_tilde}). For the convenience
of readers, the proof of Lemma~\ref{lem:t-bar}  will be recapitulated below.
\begin{lemma}\label{lem:t-bar}
  Assume   Assumptions~\ref{ass:stable} and \ref{ass:bound-support} hold.
Then, $\gamma_i(\theta_i)$ and $\xi_i(\theta)$ are well defined
for every $\theta\in \R^J$. Furthermore, $(\psi^{(r)}, \psi^{(r)}_1, \ldots, \psi^{(r)}_J)$ satisfies the transform-BAR
  \begin{align}
    \label{eq:bar-b}
    Q(\theta) {{\psi}}^{(r)}(\theta) + \sum_{i\in \mathcal{J}} (\mu^{(r)}_i-\lambda_i) \xi_i(\theta) \bigl ({\psi}^{(r)}(\theta)-  {\psi}_i^{(r)}(\theta) \bigr ) =0, \quad \theta\in \R^J_-,
  \end{align}
  where $Q(\theta)$ is defined to be
  \begin{align}
    &  \label{eq:Q}
      Q(\theta)=\sum_{i\in \mathcal{J}}\bigl ( \alpha_i \gamma_i(\theta_i) + \lambda_i \xi_i(\theta) \bigr ).
  \end{align}      
\end{lemma}

\begin{proof}[Proof of Lemma~\ref{lem:t-bar}]
  First, condition
  (\ref{eq:Tp}) and Lemma~4.1 of \cite{BravDaiMiya2017} imply that
  $\gamma_i(\theta_i)$ and $\xi_i(\theta)$ are well defined for all
  $\theta\in \R^J$. We now prove that
  $\psi^{(r)}$, together with $(\psi^{(r)}_1, \ldots, \psi^{(r)}_J)$ satisfies
  (\ref{eq:bar-b}). For each $\theta\in \R^J_-$, consider
  $f_\theta(x)$ defined in (\ref{eq:f_tilde}). Because
  $\gamma_i(\theta_i)$ satisfies (\ref{eq:eta-g}), $f_\theta(x)$
  satisfies (\ref{eq:palm-kille}).  Similarly, since $\xi_i(\theta)$
  satisfies (\ref{eq:xi-g}), $f_\theta(x)$ satisfies
  (\ref{eq:palm-kills}). Since $T_{e,i}(1)$ and $T_{s,i}(1)$ have
  bounded supports for each $i\in \mathcal{J}$, we can restrict
  $(u, v)$ component of state $x=(z, u, v)$ in a bounded domain. Thus,
  $f_\theta\in \mathcal{D}$ for each $\theta\in \R^{J}_-$.
It is easy to see
  \begin{align*}
    \mathcal{A}f_\theta(x) = & -\sum_{i\in \mathcal{J}} \alpha_i \gamma_i(\theta_i) f_{\theta}(x)  - \sum_{i\in \mathcal{J}} \mu_i^{(r)}\xi_i(\theta) f_{\theta}(x)\mathbbm{1}(z_i>0)\\
                           = & -\sum_{i\in \mathcal{J}} \Big(\alpha_i \gamma_i(\theta_i)+ \mu^{(r)}_i\xi_i(\theta)\Big)f_\theta(x) + \sum_{i\in \mathcal{J}} \mu_i^{(r)} \xi_i(\theta) f_{\theta}(x)\mathbbm{1}(z_i=0)\\
                          = & -\sum_{i\in \mathcal{J}} \Big(\alpha_i \gamma_i(\theta_i)+ \lambda_i\xi_i(\theta)\Big)f_\theta(x) - \sum_{i\in \mathcal{J}}(\mu^{(r)}_i -\lambda_i)f_{\theta} (x) \\ {}\quad  
&+ \sum_{i\in \mathcal{J}} \mu_i^{(r)} \xi_i(\theta) f_{\theta}(x)\mathbbm{1}(z_i=0).
  \end{align*}
Thus,   (\ref{eq:bar-no-palm}) implies
  \begin{align}
  &  \sum_{i\in \mathcal{J}} \Big(\alpha_i \gamma_i(\theta_i)+ \lambda_i\xi_i(\theta)\Big) \E[f_\theta(X^{(r)})] +
    \sum_{i\in \mathcal{J}}(\mu^{(r)}_i -\lambda_i)\E[f_{\theta} (X^{(r)})]
\nonumber    \\
    & - \sum_{i\in \mathcal{J}} \mu_i^{(r)} \xi_i(\theta) \E[ f_{\theta}(X^{(r)})\mathbbm{1}(Z^{(r)}_i=0)]=0. \label{eq:Af}
  \end{align}
  Finally, (\ref{eq:bar-b}) follows from \eq{Af}, the definition of
  $\psi^{(r)}(\theta)$ and $\psi^{(r)}_i(\theta)$, and
  \begin{align*}
    & \mu_i^{(r)} \E[ f_{\theta}(X^{(r)})\mathbbm{1}(Z^{(r)}_i=0)]
      =     \mu_i^{(r)} \E[ f_{\theta}(X^{(r)}) \mid Z^{(r)}_i=0]
      \Prob\{Z^{(r)}_i=0\} \\
   & = \mu^{(r)}_i(1-\rho^{(r)}_i)      \psi^{(r)}_i(\theta)
    = (\mu^{(r)}_i-\lambda_i)    \psi^{(r)}_i(\theta),
  \end{align*}
  where the second equality follows from (\ref{eq:idle}).

\end{proof}

In proving Proposition~\ref{lem:ite}, one works with MGF
$\phi^{(r)}(\theta)$ for some appropriate choices of $\theta$, whereas
transform-BAR \eq{bar-b} is stated in terms $\psi^{(r)}(\theta)$ and
$\psi^{(r)}_i(\theta)$.  The following lemma implies that working with
$\phi^{(r)}$ is equivalent to working with
$\psi^{(r)}$.  In the rest of this paper, we adopt the standard notation
that $f(r)=o(g(r))$ means $f(r)/g(r)\to 0 $ as $r\to 0$.
\begin{lemma}[SSC]\label{lem:ssc}
  Assume $\theta{(r)}\in \R^J_-$  satisfying
  \begin{align}
    \label{eq:thetaor}
     \abs{\theta{(r)}}\le c \, r, \quad r\in (0, 1)
  \end{align}
  for some constant $c>0$. Then,
\begin{align}
  & \phi^{(r)}\big(\theta{(r)}\big)-\psi^{(r)}\big(\theta{(r)}\big) =o(1)\quad \text{ as } r\to 0,
    \label{eq:ssc-phi} \\
    & \phi^{(r)}_j\big(\theta{(r)}\big)-\psi^{(r)}_j\big(\theta{(r)}\big) = o(1) \quad \text{ as } r\to 0, \quad j\in \mathcal{J},    \label{eq:ssc-phij} 
\end{align}
where $\phi^{(r)}_j(\theta)= \E[ \exp{(\sum_{j\in \mathcal{J}}\theta_jZ_j^{(r)})} \mid Z_j^{(r)}=0]$ is the MGF of $Z^{(r)}$ conditioning on $Z^{(r)}_j=0$.
\end{lemma}

Lemma~\ref{lem:ssc} holds for any sequence $\theta(r)\in \R^J_-$
satisfying \eq{thetaor}.
To motivate the term SSC used in Lemma~\ref{lem:ssc},
	recall that  the steady-state variable is
	$X^{(r)}=(Z^{(r)},R^{(r)}_e, R^{(r)}_s )$.  Lemma~\ref{lem:ssc} says
	that the component $(R^{(r)}_e, R^{(r)}_s)$ in the state variable is negligible (or
	collapsing to zero) in the sense that
	\begin{align}\label{eq:ssc-a}
		\sum_{i\in {\cal J}} \gamma_i(\theta_i{(r)}) R^{(r)}_{e,i}  +
		\sum_{i\in {\cal J}}\xi_i(\theta{(r)})  R^{(r)}_{s,i}  \to 0
	\end{align}
	almost surely because
	$\gamma_i(\theta_i(r))\to 0$
	and  $\xi_i(\theta(r))\to 0$ (see Lemma~\ref{lem:taylor} below), and
	$(R^{(r)}_e, R^{(r)}_s)$ has bounded support.

 To prove Proposition~\ref{lem:ite}, we need
to prove (\ref{eq:ite}) for any vector $\eta\in \R^J_-$ and any
station $j\in \mathcal{J} $. In the rest of this section, we fix a
vector $\eta\in \R^J_-$ and a station $j\in \mathcal{J} $.  We will
apply Lemma~\ref{lem:ssc} for two specific choices of $\theta(r)$:
$\theta^{(r)}$ and $\tilde{\theta}^{(r)}$.  Recall $(w_{ij})$ defined in (\ref{eq:ww1}) and (\ref{eq:ww2}).  We define 
  $\theta^{(r)}$ to be $\theta=(\theta_1, \ldots, \theta_J)'$ with the components being
  given as the following:
 \begin{align}
   & \theta_k =  \eta_k r^k, \quad k =j+1, \ldots, J, \label{eq:w1}\\
   & \theta_j = \eta_j r^j + \frac{1}{1-w_{jj}} \sum_{i=j+1}^J w_{ji} \theta_i,  \label{eq:w2}\\
   & \theta_\ell = w_{\ell j} \theta_j +  \sum_{i=j+1}^{J}w_{\ell i} \theta_i, \quad \ell=1, \ldots, j-1. \label{eq:w3}
 \end{align}
Similarly, 
 	we define $\tilde{\theta}^{(r)}$ to be 
 	$\tilde{\theta}=(\tilde{\theta}_1, \ldots, \tilde{\theta}_J )'$  with the components being given as the following:
 	\begin{align}
 		& \tilde{\theta}_k= \eta_k r^k, \quad k =j+1, \ldots, J, \label{eq:wt1}\\
 		&\tilde{\theta}_j = \eta_j r^{j+1/2} + \frac{1}{1-w_{jj}} \sum_{i=j+1}^J w_{ji}\tilde{\theta}_i, \label{eq:wt3}\\
 		& \tilde{\theta}_\ell = w_{\ell j} \tilde{\theta}_j +  \sum_{i=j+1}^{J}w_{\ell i} \tilde{\theta}_i, \quad \ell=1, \ldots, j-1.      \label{eq:wt5} 
 	\end{align}
        In the rest of this section, we use $\theta^{(r)}$ to denote
        the $\theta$ defined by \eq{w1}-\eq{w3}, and
        $\tilde{\theta}^{(r)}$ to denote the $\tilde{\theta}$ defined
        by (\ref{eq:wt1})-(\ref{eq:wt5}), omitting the dependence on
        $\eta\in \R^J_-$ and the station index $j\in\mathcal{J}$.  The
        reason for this particular form of $\theta^{(r)}$ and
        $\tilde{\theta}^{(r)}$ will be clear from the proof
        of Proposition~\ref{lem:ite} below.
 The uniform moment bound  (\ref{eq:Jmoment}) implies the following lemma.
 \begin{lemma}\label{lem:ssc-phi}
   Let $\theta^{(r)}$ and $\tilde{\theta}^{(r)}$ be defined in \eq{w1}-\eq{w3} and \eq{wt1}-\eq{wt5}, respectively. As $r\to 0$,
 \begin{align}
   & \phi^{(r)}(\theta^{(r)}) -\phi^{(r)}(0,\ldots, 0, r^j\eta_j, \ldots, r^J\eta_J) =o(1), \label{eq:ssc1} \\
   & \phi^{(r)}(\tilde{\theta}^{(r)}) -\phi^{(r)}(0,\ldots, 0, 0, r^{j+1}\eta_{j+1}, \ldots, r^J\eta_J) =o(1), \label{eq:ssc3}\\
   & \phi^{(r)}_j(\theta^{(r)}) -\phi^{(r)}_j(0,\ldots, 0, r^{j}\eta_{j}, \ldots, r^J\eta_J) =o(1), \label{eq:ssc2}\\
       & \phi^{(r)}_j(\tilde{\theta}^{(r)}) -\phi^{(r)}_j(0,\ldots, 0,  0, r^{j+1}\eta_{j+1}, \ldots, r^J\eta_J) =o(1). \label{eq:ssc4}       
 \end{align}   
 \end{lemma}

Both Lemma~\ref{lem:ssc} and Lemma~\ref{lem:ssc-phi} will be proved in Appendix \ref{sec:ssc} under the heading of state space collapse
(SSC) results. We will use transform-BAR \eq{bar-b} to prove
Proposition~\ref{lem:ite}. In addition to replacing $\psi^{(r)}$ by $\phi^{(r)}$ following Lemma~\ref{lem:ssc}, we need
Taylor expansions in the following lemma.
\begin{lemma}\label{lem:taylor}
 As
$\theta\to 0$, 
\begin{align}
 & \gamma_i(\theta_i) = \theta_i+ \frac{1}{2} \gamma^*_i(\theta_i) + o(\theta_i^2), \label{eq:t1}\\
 & \xi_i(\theta) = \bar{\xi}_i(\theta)  + \frac{1}{2}\xi^*_i(\theta) + o(\abs{\theta}^2),\label{eq:t2}
\end{align}
where   $\abs{\theta}=\sum_{i}\abs{\theta_i}$, and
\begin{align}
  &  \gamma_i^*(\theta_i) = c_{e,i}^2\theta_i^2, \label{eq:eta-s}   \\
  & \bar{\xi}_i(\theta) = \sum_{j\in \mathcal{J}}P_{ij}\theta_j - \theta_i, \label{eq:xi-b}\\
		& \xi^*_i(\theta) = c_{s,i}^2\Big ( \sum_{j\in \mathcal{J}}P_{ij}\theta_j - \theta_i \Big )^2 + \sum_{j\in \mathcal{J}}P_{ij}\theta_j^2 - \Big (\sum_{j\in \mathcal{J}}P_{ij}\theta_j \Big )^2.\label{eq:xi-s} 
\end{align}
 $c^2_{e,i}$ and $c^2_{s,i}$ are
the variances of the unitized random variables $T_{e,i}$ and
$T_{s,i}$, respectively. 
\end{lemma}
 In (\ref{eq:eta-s}) and (\ref{eq:xi-s}), $c^2_{e,i}$ and $c^2_{s,i}$  are also equal to the squared
coefficients of variation (SCVs) of the corresponding interarrival and
service time distributions.
When the primitive distributions are exponential, $\gamma_i(\theta_i)$ and $\xi_i(\theta)$ have the following explicit expressions: for each $\theta\in \R^J$ and $i\in \mathcal{J}$
\begin{align}
  & \gamma_i(\theta_i) = e^{\theta_i}-1, \label{eq:eta}\\
  & \xi_i(\theta) =P_{i0}\big(e^{-\theta_i}-1\big) + \sum_{j\in \cal{J} }P_{ij}\big(e^{\theta_j-\theta_i}-1\big).\label{eq:xi}
\end{align}
A direct Taylor expansion of the functions in (\ref{eq:eta}) and
(\ref{eq:xi}) is consistent with the expansions in (\ref{eq:t1}) and
(\ref{eq:t2}) with $c^2_{e,i}=1$ and $c^2_{s, i}=1$.  Although
$\gamma_i(\theta_i)$ and $\xi_i(\theta)$ are defined implicitly in
general through (\ref{eq:eta-g}) and (\ref{eq:xi-g}), the Taylor
expansions \eq{eta-s} and \eq{xi-s} are what we need to access these
two implicit functions.
We delay the proof of  Lemma~\ref{lem:taylor} to Section~\ref{sec:abar} by considering the proof of Lemma~\ref{lem:taylor} as a special case
of the proof of Lemma~\ref{lem:taylor-psi}.

\begin{proof}[Proof of Proposition~\ref{lem:ite}]
  We now prove \eq{ite} and hence the proposition   under Assumption~\ref{ass:bound-support}.
  Fix an $\eta\in \R^J_-$ and fix a station $j\in {\cal J}$.  Recall the definitions of $\theta^{(r)}$ and $\tilde{\theta}^{(r)}$
  in \eq{w1}-\eq{w3} and \eq{wt1}-\eq{wt5}, respectively.
 We will prove
 \begin{align}
   & \psi^{(r)}(\theta^{(r)})-\frac{1}{1-d_j\eta_j} \psi^{(r)}_j(\theta^{(r)}) = o(1) \quad \text{ as } r\to 0, \quad j\in {\cal J} \label{eq:bar1} \\
&  \psi^{(r)}(\tilde{\theta}^{(r)})- \psi^{(r)}_j(\tilde{\theta}^{(r)})=o(1)  \quad \text{ as } r\to 0, \quad j\in {\cal J}\setminus \{J\}.\label{eq:bar2}
 \end{align}
 Assuming (\ref{eq:bar1})-(\ref{eq:bar2}), we now complete the proof of
 Proposition~\ref{lem:ite}.

 It follows from  (\ref{eq:bar1}),  SSC  results (\ref{eq:ssc-phi})-(\ref{eq:ssc-phij}), 
 (\ref{eq:ssc1}), and (\ref{eq:ssc2}) that for each $j\in {\cal J}$
  \begin{align}
   \label{eq:ite2}
   & \phi^{(r)}(0,\ldots, 0, r^j\eta_j, \ldots, r^J\eta_J )-\frac{1}{1-d_j\eta_j} \phi^{(r)}_j( 0,\ldots, 0, r^j\eta_j, \ldots, r^J\eta_J)\to 0.
 \end{align}
 For $j=J$, $ \phi^{(r)}_J( 0,\ldots, 0,  r^J\eta_J)=1$  by definition,
and  hence \eq{ite2} implies \eq{ite}, proving Proposition~\ref{lem:ite} for $j=J$. For $j\in {\cal J}\setminus\{J\}$,
 it follows from  (\ref{eq:bar2}),
  SSC  results (\ref{eq:ssc-phi})-(\ref{eq:ssc-phij}), 
 (\ref{eq:ssc3}), and (\ref{eq:ssc4}) that
 \begin{align}\label{eq:ite3}
   & \phi^{(r)}(0,\ldots, 0, r^{j+1}\eta_{j+1}, \ldots, r^J\eta_J )- \phi^{(r)}_j( 0,\ldots, 0, r^{j+1}\eta_{j+1}, \ldots, r^J\eta_J)\to 0.
 \end{align}
By definition, $ \phi^{(r)}_j( 0,\ldots, 0, r^j\eta_j, \ldots, r^J\eta_J)=
\phi^{(r)}_j( 0,\ldots, 0, r^{j+1}\eta_{j+1}, \ldots, r^J\eta_J)$. Therefore, 
(\ref{eq:ite2}) and (\ref{eq:ite3}) imply (\ref{eq:ite}). Thus, we have proved Proposition~\ref{lem:ite}, assuming (\ref{eq:bar1})-(\ref{eq:bar2}) hold.

We now  use transform-BAR (\ref{eq:bar-b}) to prove
(\ref{eq:bar1})-(\ref{eq:bar2}). We first prove \eq{bar1} for $j\in {\cal J}$.
Recall the definitions of $\bar{\xi}_i$, $\gamma^*_i$,
and $\xi^*_i$ in (\ref{eq:eta-s})-(\ref{eq:xi-s}).
It  follows from traffic equation (\ref{trafficeq}) that
\begin{align}
  \sum_{i\in \mathcal{J}} \bigl ( \alpha_i \theta_i + \lambda_i \bar{\xi}_i(\theta) \bigr )=0 \quad \text{ for each } \theta\in \R^J. \label{eq:flow-balance}
\end{align}
By Taylor expansions (\ref{eq:t1})-(\ref{eq:t2}), the $Q$ in (\ref{eq:Q}) has the following Taylor expansion
\begin{align}
  Q(\theta) & =   \sum_{i\in \mathcal{J}} \bigl ( \alpha_i \theta_i + \lambda_i \bar{\xi}_i(\theta) \bigr )+ \frac{1}{2} \sum_{i\in \mathcal{J}} \big (\alpha_i \gamma_i^*(\theta_i) + \lambda_i \xi_i^*(\theta)\big) +  o(\abs{\theta}^2) \nonumber \\
 & =  Q^*(\theta) + o(\abs{\theta}^2) \quad \text{ as } \theta\to 0,   \label{eq:taylor-Q}
\end{align}
where
\begin{align}
        & Q^*(\theta) = \frac{1}{2} \sum_{i\in \mathcal{J}} \big (\alpha_i \gamma_i^*(\theta_i) + \lambda_i \xi_i^*(\theta)\big). \label{eq:Q-s}
\end{align}
As  a result, the transform-BAR (\ref{eq:bar-b}) implies that
 for any sequence  $\theta=\theta{(r)}\in \R^J_-$ with $\theta{(r)}\to 0$  as $r\to 0$,
\begin{align}
  &    Q^*(\theta)\psi^{(r)}(\theta)+ \sum_{i\in \mathcal{J} } (\mu^{(r)}_i - \lambda_i ) \bar{\xi}_i(\theta) \bigl ( \psi^{(r)}(\theta)-\psi^{(r)}_i(\theta) \bigr ) \nonumber \\
  & =-\frac{1}{2}\sum_{i\in\mathcal{J}} (\mu^{(r)}_i - \lambda_i ) {\xi}^*_i(\theta) \bigl ( \psi^{(r)}(\theta)-\psi^{(r)}_i(\theta) \bigr ) + o(\abs{\theta}^2).    \label{eq:abar-1}
\end{align}



Recall $\theta^{(r)}$ defined in (\ref{eq:w1})-(\ref{eq:w3}).
One can verify
that
\begin{align}
  &  \bar{\xi}_i(\theta^{(r)})=0 \quad i=1, \ldots, j-1,\label{eq:pb1}\\
  &  \bar{\xi}_j(\theta^{(r)})=(1-w_{jj}) r^j \eta_j;\label{eq:pb2}
\end{align}
see Lemma~\ref{lem:linear} for a proof.
From the definition of  $\theta^{(r)}$,
\begin{align*}
  &\theta^{(r)}_\ell=w_{\ell j} r^j\eta_j +o(r^j), \quad \ell =1, \ldots, j-1, \\
  & \theta^{(r)}_j = r^j\eta_j +o(r^j), \\
  &  \theta^{(r)}_k = o(r^j), \quad k=j+1, \ldots, J.
\end{align*}
Therefore,
under the  multi-scale heavy traffic condition (\ref{htpar}),
\begin{align}
  & (\mu^{(r)}_i-\lambda_i) {\xi}^*_i(\theta^{(r)}) = o(r^{2j})  \quad  i\in \mathcal{J}, \label{eq:pb3}\\
  & (\mu^{(r)}_i-\lambda_i) \bar{\xi}_i(\theta^{(r)}) = o(r^{2j}) \quad i=j+1,\ldots, J.\label{eq:pb4}
\end{align}
It follows from (\ref{eq:pb1})-(\ref{eq:pb4}) and (\ref{eq:abar-1}) that
\begin{align}\label{eq:abar-j}
  &    Q^*\big(\theta^{(r)}\big)\psi^{(r)}\big(\theta^{(r)}\big) + (1-w_{jj})r^{2j} \eta_j
    \Bigl ( \psi^{(r)}\big(\theta^{(r)}\big)-\psi^{(r)}_j\big(\theta^{(r)}\big) \Bigr ) = o(r^{2j}).
\end{align}
It is easy to check that 
\begin{align}
  & \lim_{r\to 0} \frac{Q^*(\theta^{(r)})}{r^{2j}} = \eta_j^2 Q^*(w_{1j},\ldots, w_{j-1,j}, 1, 0, \ldots, 0),
    \label{eq:Qlimit1} 
\end{align}
which is equal to $\frac{1}{2}\sigma^2_j \eta_j^2$ by  Lemma~\ref{lem:sigma}.
Dividing both sides of (\ref{eq:abar-j}) by $r^{2j}$, we have
\begin{align*}
  (\sigma^2_j/2) \eta_j^2\psi^{(r)}(\theta^{(r)}) + (1-w_{jj}) \eta_j  \bigl ( \psi^{(r)}(\theta^{(r)})-\psi^{(r)}_j(\theta^{(r)}) \bigr ) = o(1),
\end{align*}
from which (\ref{eq:bar1}) follows.

We next prove (\ref{eq:bar2}) for $j\in {\cal J}\setminus \{J\}$.
Recall $\tilde{\theta}^{(r)}$ defined in
(\ref{eq:wt1})-(\ref{eq:wt5}). From the definition of 
$\tilde{\theta}^{(r)}$,
	  \begin{align*}
		&\tilde{\theta}^{(r)}_\ell=w_{\ell j} r^{j+1/2}\eta_j +o(r^{j+1/2}), \quad \ell =1, \ldots, j-1, \\
		& \tilde{\theta}^{(r)}_j = r^{j+1/2}\eta_j +o(r^{j+1/2}), \\
		&  \tilde{\theta}^{(r)}_k = o(r^{j+1/2}), \quad k=j+1, \ldots, J.
          \end{align*}
          Therefore, $\bar{\xi}_i(\tilde{\theta}^{(r)})=O(r^{j+1/2})$ and ${\xi}^*_i(\tilde{\theta}^{(r)})=O(r^{2j+1})$ for $i\in \mathcal{J}$.
	Using  (\ref{eq:lin1}) and  multi-scale heavy traffic condition (\ref{htpar}), the transform-BAR \eq{bar-b} leads to
	\begin{align*}
		&    Q^*(\tilde{\theta}^{(r)})\psi^{(r)}(\tilde{\theta}^{(r)}) + r^j\bar{\xi}_j(\tilde{\theta}^{(r)})
		\bigl ( \psi^{(r)}(\tilde{\theta}^{(r)})-\psi^{(r)}_j(\tilde{\theta}^{(r)}) \bigr ) = o(r^{2j+1}).
	\end{align*}
	Using (\ref{eq:lin3}), the preceding equation becomes
	\begin{align*}
		&    Q^*(\tilde{\theta}^{(r)})\psi^{(r)}(\tilde{\theta}^{(r)}) + (1-w_{jj})r^{2j+1/2} \eta_j
		\bigl ( \psi^{(r)}(\tilde{\theta}^{(r)})-\psi^{(r)}_j(\tilde{\theta}^{(r)}) \bigr ) = o(r^{2j+1}).
	\end{align*}        
        Similar to \eq{Qlimit1}, one has 
	\begin{align*}
	\lim_{r\to 0} \frac{Q^*(\tilde{\theta}^{(r)})}{r^{2j+1}} = \eta_j^2 Q^*(w_{1j}, \ldots, w_{j-1,j}, 1, 0, \ldots, 0).   
	\end{align*}
	Dividing both sides by $r^{2j+1/2}$ and using (\ref{eq:Qlimit3}), we have
	\begin{align*}
		\sqrt{r} (\sigma^2_j/2) \eta_j^2\psi^{(r)}(\tilde{\theta}^{(r)})  +    (1-w_{jj}) \eta_j  \bigl ( \psi^{(r)}(\tilde{\theta}^{(r)})-\psi^{(r)}_j(\tilde{\theta}^{(r)}) \bigr ) = o(\sqrt{r}),
	\end{align*}
	from which (\ref{eq:bar2}) follows.

\end{proof}

This completes the proof of Proposition~\ref{lem:ite} and hence \thm{main} under the bounded support assumption. In Section~\ref{sec:abar}, we  proceed to a general proof of
Proposition~\ref{lem:ite}, using a standard truncation argument.

\section{Asymptotic BAR and the truncation arguments}
\label{sec:abar}
As outlined at the beginning of Section~\ref{sec:outline}, to prove
Theorem~\ref{thm:main}, it suffices to prove
Proposition~\ref{lem:ite}.
Section~\ref{sec:proof-b} provided a proof of
Proposition~\ref{lem:ite}  under Assumption~\ref{ass:bound-support}
(bounded supports).  This section details the changes needed to prove
Proposition~\ref{lem:ite} without
Assumption~\ref{ass:bound-support}. 
The key is to develop an
asymptotic version of the transform-BAR (\ref{eq:abar}) to replace
(\ref{eq:abar-1}); the latter was  used for  the proof of
Proposition~\ref{lem:ite} in Section~\ref{sec:proof-b}.

Throughout this section, we assume
Assumptions~\ref{ass:mscale}-\ref{ass:stable}.
When the unitized interarrival and service times $T_{e,i}(1)$ and
$T_{s,i}(1)$ are bounded, we have employed test function $f_\theta(x)$ in
(\ref{eq:f_tilde}) in Section~\ref{sec:proof-b}, where $\gamma_i(\theta_i)$ and $\xi_i(\theta)$
satisfy (\ref{eq:eta-g}) and (\ref{eq:xi-g}) in order for
$f_{\theta}(x)$ to satisfy (\ref{eq:palm-kille}) and
(\ref{eq:palm-kills}). Without assuming Assumption~\ref{ass:bound-support}, we will work with the
truncated random variables $T_{e,i}(1)\wedge r^{-1}$ and  $T_{s,i}(1)\wedge r^{-1}$ for $r\in (0, 1)$, where $a\wedge b\equiv \min(a, b)$ for $a, b\in \R$.
Equivalently, for each $\theta\in \R^J_-$ and each $r\in (0,1)$, we wish to work
with  test  function $f_{(\theta, r)}(x)$ for $x=(z, u, v)\in \Z^J_+\times \R^J_+\times \R^J_+$ via
\begin{align}
  \label{eq:test_t}
  & f_{(\theta, r)} (x) = \exp\Bigl ( \langle \theta, z\rangle -
    \sum_{j} \gamma_j(\theta_j, r) (\alpha_j u_j\wedge r^{-1})
    - \sum_{j} \xi_j(\theta, r) (\mu_j^{(r)}v_j\wedge r^{-1})
    \Bigr ).
\end{align}
In order for $f_{(\theta, r)}$ to satisfy (\ref{eq:palm-kille}) and
(\ref{eq:palm-kills}), $\gamma_i(\theta_i, r)$ and $\xi_i(\theta, r)$
need to satisfy
\begin{align}
  & e^{\theta_j} \E[e^{-\gamma_j(\theta_j, r) (T_{e,j}(1)\wedge r^{-1})}]=1,\label{eq:eta_t} \\
  & \bigl ( \sum_{k\in \mathcal{J}} P_{jk}e^{\theta_k-\theta_j}+ P_{j0}e^{-\theta_j} \bigr ) \E[e^{-\xi_j(\theta,r) (T_{s,j}(1)\wedge r^{-1})}]=1.\label{eq:xi_t}
\end{align}
Without condition (\ref{eq:Tp}) in Assumption~\ref{ass:bound-support},
$\gamma_i(\theta_i, r)$ and $\xi_i(\theta, r)$ are not guaranteed to
be well defined for every $\theta\in \R^J_-$.  Lemma 3.2 of
\cite{BravDaiMiya2017} says that when $\abs{\theta}<M$ for small
enough constant $M>0$, $\gamma_j(\theta_j, r)$ and $\xi_j(\theta, r)$
are defined for each $r\in (0, 1)$.  Clearly, when the interarrival
and service times are bounded by $1/r$,
$\gamma_j(\theta_j,r)=\gamma_j(\theta_j)$ and
$\xi_j(\theta, r)=\xi_j(\theta)$. Therefore, the proof of
Proposition~\ref{lem:ite} in Section~\ref{sec:proof-b} is a special
case of the proof presented below.  Most readers will find it easier
to first read the proof in Section~\ref{sec:proof-b} before proceeding
further.

In the following,  define $\Theta=\{\theta \in \R^J_-: \abs{\theta}<M\}$.
For each $\theta\in \Theta$,  define 
	\begin{align} \label{eq:psi}
	{\psi}^{(r)}(\theta) = \E[ f_{(\theta, r)}(X^{(r)})], \quad
	{\psi}^{(r)}_j(\theta) = \E\big[ f_{(\theta, r)}(X^{(r)}) \mid Z^{(r)}_j=0\big].
	\end{align}
        The symbol $\psi^{(r)}$ is identical to the one used in
        (\ref{eq:psi0}). The reusage of the symbol is
        intentional. First, when the primitive distributions have
        bounded supports, the two definitions are indeed identical
        when $r$ is small enough. Second, the $\psi^{(r)}$ in this
        section always follows the definition in (\ref{eq:psi}), but
        the same symbol reminds us to trace expressions and equations
        involving $\psi^{(r)}$ in Section~\ref{sec:proof-b}, which could be
        copied verbatim into this section. This fact allows us to
        focus on those aspects that affect current proof differently.

        We first establish three lemmas that are analogs to
        those in Section~\ref{sec:proof-b}. 
        The following lemma replaces
        Lemma~\ref{lem:t-bar}, and the resulting equation (\ref{eq:bar-g}) is the asymptotic version of the transform-BAR (\ref{eq:bar-b}).
	
%
%

\begin{lemma}	\label{lem:bar-g}
  Assume   Assumption~\ref{ass:stable} holds. Then, for each $r\in (0, 1)$,
  \begin{align}
    \label{eq:bar-g}
    Q^{(r)}(\theta) {{\psi}}^{(r)}(\theta) + \sum_{i\in \mathcal{J}} (\mu^{(r)}_i-\lambda_i) \xi_i(\theta, r) \bigl ({\psi}^{(r)}(\theta)-  {\psi}_i^{(r)}(\theta) \bigr ) 
		= \epsilon(\theta, r), \quad \theta\in \Theta,
		\end{align}
		where
		\begin{align}
		{Q}^{(r)}(\theta) &=  \sum_{i\in \mathcal{J}}\Big(\alpha_i \gamma_i(\theta_i, r)+ \lambda_i \xi_i(\theta, r) \Big),\label{eq:Q-t}     \\
		\epsilon(\theta, r) &  = \sum_{i\in \mathcal{J}} \alpha_i\gamma_i(\theta_i, r) \E\big[ \mathbbm{1}(\alpha_i R^{(r)}_{e, i}> 1/r) f_{(\theta,r)}(X^{(r)})\big]\nonumber \\
		& + \sum_{i\in \mathcal{J}}  \mu^{(r)}_i \xi_i (\theta, r)
		\E\big[ \mathbbm{1}(\mu_i^{(r)}R^{(r)}_{s, i}> 1/r, Z^{(r)}_i>0)f_{(\theta,r)}(X^{(r)})\big], \label{eq:epsilon}
		\end{align}  
		and $\gamma_i(\theta_i, r)$, $\xi_i(\theta, r)$ are defined in \eq{eta_t} and \eq{xi_t}, respectively.
	\end{lemma}
\begin{proof}
  For each $\theta\in \Theta$ and $r\in (0,1)$,
  $\gamma_i(\theta_i,r)$ and $\xi_i(\theta, r)$ are well
  defined by Lemma~3.2 of \cite{BravDaiMiya2017}. It is easy to check that  $f_{(\theta,r)}\in \mathcal{D}$, and $f_{(\theta, r)}$
          satisfies (\ref{eq:palm-kille}) and
          (\ref{eq:palm-kills}).  It follows from (\ref{eq:bar-no-palm}) that
          $\E[f_{(\theta, r)}(X^{(r)})]=0$. Note that
          \begin{align*}
            \mathcal{A}f_{(\theta, r)}(x) = & -\sum_{i\in \mathcal{J}} \alpha_i \gamma_i(\theta_i, r) \mathbbm{1}(\alpha_iu_i \le r^{-1})f_{(\theta, r)}(x)  \\
                                            & - \sum_{i\in \mathcal{J}} \mu_i^{(r)}\xi_i(\theta, r) \mathbbm{1}(\mu^{(r)}_iv_i\le r^{-1}) f_{(\theta, r)}(x)\mathbbm{1}(z_i>0) \\
            = &  -\sum_{i\in \mathcal{J}} \alpha_i \gamma_i(\theta_i, r) f_{(\theta,r)}(x)  - \sum_{i\in \mathcal{J}} \mu_i^{(r)}\xi_i(\theta, r) f_{(\theta,r)}(x)\mathbbm{1}(z_k>0)\\
                                            & +\sum_{i\in \mathcal{J}} \alpha_i \gamma_i(\theta_i, r) \mathbbm{1}(\alpha_iu_i > r^{-1})f_{(\theta, r)}(x) \\
            & + \sum_{i\in \mathcal{J}} \mu_i^{(r)}\xi_i(\theta, r) \mathbbm{1}(\mu^{(r)}_iv_i> r^{-1}) f_{(\theta, r)}(x)\mathbbm{1}(z_i>0)            
          \end{align*}
          The rest of the proof is identical to the proof of Lemma~\ref{lem:t-bar} in Section~\ref{sec:proof-b}.
              
	\end{proof}

The next lemma  replaces Lemma~\ref{lem:ssc}. Actually,
        Lemma~\ref{lem:ssc} is a special case of Lemma~\ref{lem:ssc-unbound},
 which  will be proved  in
  Appendix~\ref{sec:ssc}. The statements of these two lemmas are almost identical, except that $\theta(r)$ is required to be in $\Theta$ in Lemma~\ref{lem:ssc-unbound}.

\begin{lemma}\label{lem:ssc-unbound}
  For
  $(\psi^{(r)}, \psi^{(r)}_1, \ldots,
  \psi^{(r)}_J)$ defined in (\ref{eq:psi}), the SSC results continue to  hold. Namely, as $r\to 0$,
  \begin{align}
    & \phi^{(r)}(\theta(r))-\psi^{(r)}(\theta(r)) =o(1), \quad
      \label{eq:ssc-phi0}\\
    &  \phi_j^{(r)}(\theta(r))-\psi^{(r)}_j(\theta(r)) =o(1) \quad \text{ for } j\in \mathcal{J}
      \label{eq:ssc-phii}
  \end{align}
  for any sequence $\theta(r)\in \Theta$ satisfying
  (\ref{eq:thetaor}).          
\end{lemma}


        The next lemma provides Taylor expansions for $\gamma_i(\theta_i, r)$ and $\xi_i(\theta, r)$, replacing Taylor expansions  for $\gamma_i(\theta_i)$ and $\xi_i(\theta)$ in (\ref{eq:t1}) and (\ref{eq:t2})  in Lemma~\ref{lem:taylor}.
The proof is  analogous to the proof of  Lemma 7.5 in \cite{BravDaiMiya2024}.
	For completeness, we provide a full proof in Appendix~\ref{app:C}.
	\begin{lemma}\label{lem:taylor-psi}
          Let  $\theta(r)\in \Theta$,  $r\in (0,1)$,            satisfying (\ref{eq:thetaor}).
                     Denote  $\theta=\theta(r)$. Then, 
                     as $r\to 0$,
		\begin{align}
		\label{eq:eta-t-taylor}
		&  \gamma_i(\theta_i, r) = \theta_i + \frac{1}{2}\gamma_i^*(\theta_i) + o(r^{J}\theta_i) + o(\theta_i^2),  \quad i\in \mathcal{J} \\
		&    \xi_i(\theta, r) = \bar{\xi}_i(\theta) + \frac{1}{2}\xi_i^*(\theta) + o(r^{J}\abs{\theta}) + o(\abs{\theta}^2)  \quad i\in \mathcal{J}.       \label{eq:xi-t-taylor}
		\end{align}      
	\end{lemma}

 We now complete the proof of Proposition~\ref{lem:ite} without Assumption~\ref{ass:bound-support}.
              
 \begin{proof}[Proof of Proposition~\ref{lem:ite} (without Assumption~\ref{ass:bound-support})]
  Assume Assumptions~\ref{ass:mscale}-\ref{ass:stable} all hold.   
  Examine the proof  of Proposition~\ref{lem:ite} in Section~\ref{sec:proof-b} under the bounded
  support Assumption~\ref{ass:bound-support}. With Lemma~\ref{lem:ssc-unbound} replacing Lemma~\ref{lem:ssc}, the remaining step
    is to prove equations (\ref{eq:bar1}) and (\ref{eq:bar2})
    for 
  $(\psi^{(r)}, \psi^{(r)}_1, \ldots, \psi^{(r)}_J)$ defined in
  (\ref{eq:psi}).

 To prove (\ref{eq:bar1}) and (\ref{eq:bar2}) for
  $(\psi^{(r)}, \psi^{(r)}_1, \ldots, \psi^{(r)}_J)$ defined in
  (\ref{eq:psi}),  the key is to extend
  asymptotic BAR (\ref{eq:abar-1}) to an asymptotic version that is appropriate for  the current setting:
 as $r\to 0$, 
		\begin{align}
		&    Q^*(\theta)\psi^{(r)}(\theta)+ \sum_{i\in \mathcal{J}} (\mu^{(r)}_i-\lambda_i) \bar{\xi}_i(\theta) \bigl ( \psi^{(r)}(\theta)-\psi^{(r)}_i(\theta) \bigr ) \nonumber \\
		& =-\frac{1}{2}\sum_{i\in\mathcal{J}} (\mu^{(r)}_i-\lambda_i) {\xi}^*_i(\theta) \bigl ( \psi^{(r)}(\theta)-\psi^{(r)}_i(\theta) \bigr ) + o(r^J\abs{\theta}) + o(\abs{\theta}^2)      \label{eq:abar}
		\end{align}
		for any sequence  $\theta=\theta(r)$ satisfying (\ref{eq:thetaor}).
  

To prove (\ref{eq:abar}), our starting point is (\ref{eq:bar-g}) in Lemma~\ref{lem:bar-g}. We first prove
\begin{align}\label{eq:epsilonrJ}
  \epsilon(\theta(r), r) = o(\abs{\theta(r)}r^J),
\end{align}
where $\epsilon(\theta, r)$ are defined in \eq{epsilon}. To prove (\ref{eq:epsilonrJ}), we first claim
\begin{align}
  &  \Prob\big\{ \alpha_j R^{(r)}_{e,j}>1/r\big\} =o(r^{J}), \label{eq:Re-t} \\
  & \Prob\big\{ \mu^{(r)}_j R^{(r)}_{s,j}>1/r\big\}  =o(r^{J}).\label{eq:Rs-t}
\end{align}
Indeed, 
\begin{align*}
  \Prob\{\alpha_j R^{(r)}_{e, j}> 1/r\} & \le r^{J}\E[(\alpha_j R^{(r)}_{e, j})^{J}\mathbbm{1}(\alpha_j R^{(r)}_{e, j}> 1/r)] \\
                                        & =
                                          \frac{r^J}{J+1} \E\big[ \big(T_{e,j}(1)\big)^{J+1} \mathbbm{1}(T_{e,j}(1)> r^{-1})\big] =o(r^J),
\end{align*}
where the first equality follows from Lemma 6.4 of \cite{BravDaiMiya2024},
and the second equality follows from condition \eq{moment}.
Equation \eq{Rs-t} can be proved similarly.

We next claim there exists a $r_0\in (0,1)$ and $M>0$ such that
\begin{align}\label{eq:fbound}
  & \sup_{r\in (0, r_0)} \sup_{x=(z,u, v) }f_{(\theta(r),r)}(x) \le M.
\end{align}
Indeed, by Lemma~\ref{lem:taylor-psi}, there exists $r_0\in (0,1)$ and $c_1>0$ such that for $r\in (0, r_0)$
\begin{align}\label{eq:xior}
  &\abs{  \gamma_i(\theta_i(r), r)}\le c_1 \abs{\theta_i(r)}\le c_1 c \, r, \quad
 \abs{  \xi_i(\theta(r), r)}\le c_1 \abs{\theta(r)}\le c_1 c \, r,
\end{align}
from which (\ref{eq:fbound}) follows. It is clear that
(\ref{eq:epsilonrJ}) follows from (\ref{eq:Re-t})-(\ref{eq:xior}).
        
To complete the proof of (\ref{eq:abar}), it remains to prove that for $\theta(r)$ satisfying (\ref{eq:thetaor}), the following holds
\begin{align*}
  & \Big( Q^{(r)}(\theta) {{\psi}}^{(r)}(\theta) + \sum_{i\in \mathcal{J}} (\mu^{(r)}_i-\lambda_i) \xi_i(\theta, r) \bigl ({\psi}^{(r)}(\theta)-  {\psi}_i^{(r)}(\theta) \bigr ) \Big) \\
  & - \Big( Q^*(\theta) {{\psi}}^{(r)}(\theta) + \sum_{i\in \mathcal{J}} (\mu^{(r)}_i-\lambda_i) \bigl(\bar{\xi}_i(\theta)+\frac{1}{2}\xi^*_i(\theta) \bigr)\bigl ({\psi}^{(r)}(\theta)-  {\psi}_i^{(r)}(\theta) \bigr ) \Big)   \\
  & = o(r^J\abs{\theta}(r)) + o(\abs{\theta(r)}^2),
\end{align*}
which follows from the Taylor expansions (\ref{eq:eta-t-taylor}) and
(\ref{eq:xi-t-taylor}), definition of $Q^{(r)}$ in (\ref{eq:Q-t}), and
(\ref{eq:flow-balance}), and the fact that $\psi^{(r)}(\theta(r))\le 1$ and $\psi^{(r)}_i(\theta(r))\le 1$  for each $r\in (0, 1)$ and $i\in \mathcal{J}$.
\end{proof}

	

	\bibliography{dai20240229}
	
	\appendix
	\begin{appendix}
		\addtocontents{toc}{\protect\setcounter{tocdepth}{1}}
		\makeatletter
		\addtocontents{toc}{%
			\begingroup
			\let\protect\l@chapter\protect\l@section
			\let\protect\l@section\protect\l@subsection
		}
		\makeatother
		
		\section{Supporting lemmas to prove Proposition~\ref{lem:ite}}
	\label{sec:lem53}
To assist the proof of Proposition~\ref{lem:ite}, we  state and prove the following three lemmas.
	\begin{lemma}\label{lem:w}
		Fix $j\in \mathcal{J}$. The following equations
		\begin{equation}\label{eq:w}
		w_{ij} = P_{ij} + \sum_{k<j} P_{ik}w_{kj},  \quad i\in\cal{J}
		\end{equation}
		have a unique solution $(w_{1j}, \ldots, w_{J,j})$ that is given by (\ref{eq:ww1}) and (\ref{eq:ww2}).
	\end{lemma}
	\begin{proof}
		The first $j-1$ equations in (\ref{eq:w}) uniquely solve
		$(w_{1j}, \ldots, w_{j-1,j})'$, which is given by (\ref{eq:ww1}).
		(\ref{eq:ww2}) is simply the vector version of the last $J-j+1$
		equations with $(w_{1j}, \ldots, w_{j-1,j})'$ being substituted by the
		right side of (\ref{eq:ww1}).
	\end{proof}
	The quantity $w_{ij}$ has the following probabilistic interpretations.
	Consider the DTMC on state space $\{0\}\cup \mathcal{J}$ corresponding to the routing matrix $P$. Let $w_{ij}$ be the probability
	that starting from state $i$, the DTMC will eventually visit state $j$, avoiding visiting states in $\{0, j+1, \ldots, J\}$. By the ``first-step'' method,  $w_{ij}$ satisfies (\ref{eq:w}). 
	
\begin{lemma}[Solution to a linear system]
  \label{lem:linear}
  Fix $\eta\in \R^J_-$ and $j\in \mathcal{J}$. Recall $\theta=\theta^{(r)}$ defined in (\ref{eq:w1})-(\ref{eq:w3})
and $\tilde{\theta}=\tilde{\theta}^{(r)}$ defined in (\ref{eq:wt1})-(\ref{eq:wt5}).
  Then $\theta$ and $\tilde{\theta}$ satisfy the following equations.
  \begin{align}
    &    \bar{\xi}_\ell(\theta)=0,  \quad    \bar{\xi}_\ell(\tilde{\theta})=0, \quad\ell =1, \ldots, j-1,  \label{eq:lin1}\\
    &  \bar{\xi}_{j}(\theta) = (1-w_{jj}) r^j \eta_j, \label{eq:lin2}\\
    &  \bar{\xi}_{j}(\tilde{\theta}) = (1-w_{jj}) r^{j+1/2} \eta_j.\label{eq:lin3}
  \end{align}
\end{lemma}
\begin{proof}
  With $\theta_k=r^k\eta_k$ for $k=j+1,\ldots, J$ being fixed, the linear system
  of equations
  \begin{align*}
    &\bar{\xi}_\ell(\theta)=0, \quad \ell =1, \ldots, j-1, \\
    &\bar{\xi}_j(\theta)= (1-w_{jj})r^j \eta_j 
  \end{align*}
  has $j$ equations and $j$ variables $(\theta_1, \ldots, \theta_j)$. It has a unique solution $(\theta_1, \ldots, \theta_j)'$ given by (\ref{eq:w2}) and (\ref{eq:w3}). Similarly,  fixing $\tilde{\theta}_k=r^k\eta_k$ for $k=j+1,\ldots, J$,
  	the linear system  of equations
  	\begin{align*}
  		&\bar{\xi}_\ell(\tilde{\theta})=0, \quad \ell =1, \ldots, j-1, \\
  		&\bar{\xi}_j(\tilde{\theta})= (1-w_{jj})r^{j+1/2} \eta_j 
  	\end{align*}
  	has a unique solution $(\tilde \theta_1, \ldots, \tilde \theta_j)'$ given by (\ref{eq:wt3}) and (\ref{eq:wt5}).
\end{proof}
	
	\begin{lemma}\label{lem:sigma}
		Fix $j\in \mathcal{J}$.
		Recall  $\sigma^2_j$ is defined in (\ref{eq:sigma}). Then,
		\begin{align}
		\sigma^2_j =2 Q^*(w_{1j},\ldots, w_{j-1,j}, 1, 0, \ldots, 0), \label{eq:Qlimit3}
		\end{align}
		where $Q^*$ is defined in \eq{Q-s}.
	\end{lemma}
	
	\begin{proof}
		Define 
		\begin{align}\label{eq:u}
		u = (w_{1j}, \ldots, w_{j-1,j}, 1, 0, \ldots, 0 )' \in \R^J,
		\end{align}
		where prime denotes transpose.
		It follows from (\ref{eq:w}) that 
		\begin{align*}
		& \sum_{k\in \mathcal{J}}P_{ik}u_k=w_{ij}, \quad i\in \mathcal{J}.
		\end{align*}
		Recall the definition $\xi^*_i(u)$ in (\ref{eq:xi-s}).
		It follows that 
		\begin{align*}
		\xi^*_i(u) & =  c^2_{s,i}\Bigl ( -u_i + \sum_{k\in \mathcal{J}} P_{i,k} u_k  \Bigr )^2 
		+  \sum_{k\in \mathcal{J}} P_{ik} u_k^2 - \big(\sum_{k\in \mathcal{J}} P_{ik} u_k\big)^2 \\
		&=  c_{s,i}^2 (-u_i+w_{ij})^2 + \sum_{k\in  \mathcal{J}} P_{ik}(u_k^2-u_k) +  w_{ij}- w_{ij}^2.                       
		\end{align*}
		For $u$ in (\ref{eq:u}), 
		\begin{align*}
		-u_i+w_{ij} =
		\begin{cases}
		0  & i< j, \\
		-1 + w_{jj}, & i=j, \\
		w_{ij} & i>j.
		\end{cases}
		\end{align*}
		Denoting
		\begin{align*}
		b_i=\sum_{k\in  \mathcal{J}} P_{ik}(u_k^2-u_k) +  w_{ij}- w_{ij}^2,
		\end{align*}
		one has the following
		\begin{align*}
		&  \sum_{i=1}^J \lambda_i\xi^*_i(u) =\lambda_j c_{s,j}^2(1-w_{jj})^2 +  \sum_{i>j} \lambda_i c_{s,i}^2 w_{ij}^2  + \sum_{i\in \mathcal{J}} \lambda_i b_i,
		\end{align*}
		where   the last term $    \sum_{i\in \mathcal{J}} \lambda_i b_i$ is equal to 
		\begin{align*}
		&  \sum_{i\in \mathcal{J}} \lambda_i \sum_{k\in \mathcal{J}} P_{ik}(u_k^2-u_k) + \sum_{i\in \mathcal{J}}\lambda_iw_{ij}(1-w_{ij}) \\
		& = \sum_{k\in \mathcal{J}} (\lambda_k-\alpha_k) (u_k^2-u_k) + \sum_{i\in \mathcal{J}}\lambda_iw_{ij}(1-w_{ij}) \\
		& =  \sum_{k<j} (\alpha_k-\lambda_k) (w_{kj}-w_{kj}^2) + \sum_{i\in \mathcal{J}}\lambda_iw_{ij}(1-w_{ij})  \\
		&=  \sum_{k<j} \alpha_k (w_{kj}-w_{kj}^2) + \sum_{i\ge j}\lambda_iw_{ij}(1-w_{ij}).
		\end{align*}
		Following  the definition of the quadratic function $Q^*(u)$ in
		(\ref{eq:Q-s}) with $\eta^*_i(u_i)$ defined in (\ref{eq:eta-s}), one has
		\begin{align*}
		2Q^*(u) &  = \alpha_j c_{e,j}^2 +  \sum_{i< j} \alpha_i c_{e,i}^2w^2_{ij}
		+ \lambda_j c_{s,j}^2(1-w_{jj})^2 +  \sum_{i>j} \lambda_i c_{s,i}^2 w_{ij}^2 \\
		& +   \sum_{i<j} \alpha_i (w_{ij}-w_{ij}^2) + \sum_{i\ge j}\lambda_iw_{ij}(1-w_{ij}) = \sigma^2_j.
		\end{align*}
	\end{proof}

	\section{State space collapse}
	\label{sec:ssc}
In this section, we  prove Lemma \ref{lem:ssc-phi} and Lemma \ref{lem:ssc-unbound}. The latter  supersedes Lemma~\ref{lem:ssc}.

	\begin{proof}[Proof of Lemma \ref{lem:ssc-unbound}]
       Let $\theta(r)\in \Theta$ be
          a sequence satisfying (\ref{eq:thetaor}). We prove
          (\ref{eq:ssc-phi0})-(\ref{eq:ssc-phii}).  In the following,
          all $\theta$ usage should be interpreted as $\theta(r)$.  We
          first prove (\ref{eq:ssc-phi0}).  By  condition \eq{thetaor} and Taylor expansion
          (\ref{eq:eta-t-taylor}),  there exists $r_0\in (0,1)$ such that
          \begin{align*}
		\abs{\gamma_i(\theta_i, r)}\le c_1 r \text{ and }  \abs{\xi_i(\theta, r)}\le c_1 r
		\quad r\in (0, r_0), i\in \mathcal{J},
		\end{align*}
		where $c_1>0$ is some constant.
		It follows that for $r\in (0, r_0)$
		\begin{align*}
                  & \abs{ \phi^{(r)}(\theta) -{\psi}^{(r)}(\theta)} \\
		& \le
		\E\Babs{1-\exp\Big(- \sum_{i} \gamma_i(\theta_i, r) (\alpha_i R^{(r)}_{e,i}\wedge r^{-1} )
			- \sum_{i} \xi_i(\theta, r) (\mu^{(r)}_i R^{(r)}_{s,i}\wedge r^{-1}) \Big)}
		\\
		& \le e^{2J c_1} c_1 \Big( \sum_{i}\alpha_i r \E[R^{(r)}_{e,i}]  + \sum_{i}  \mu^{(r)}_i r \E[R^{(r)}_{s,i}]\Big) =o(1),
		\end{align*}
		where    the second inequality follows from
		\begin{align*}
		\abs{1-e^{-x}} \le e^{\abs{x}}\abs{x}, \quad x\in \R,
		\end{align*}
		and the last equality follows from Lemma 4.5 of \cite{BravDaiMiya2017}.
		Therefore, (\ref{eq:ssc-phi0}) is proved.
		
		We next prove \eq{ssc-phii}.
		It follows that,  for $r\in (0, r_0)$,
		\begin{align*}
		& \abs{ \phi^{(r)}_j(\theta) -{\psi}^{(r)}_j(\theta)} \\
		& \le
		\E\Big[ \babs{1-\exp\big(- \sum_{i} \gamma_i(\theta, r) (\alpha_i R^{(r)}_{e,i}\wedge (1/r))
			- \sum_{i} \xi_i(\theta, r) (\mu^{(r)}_i R^{(r)}_{s,i}\wedge (1/r)) \big) } \mid Z^{(r)}_j=0\Big]\\
		& \le e^{2J c_1} c_1 \Big(
		\sum_{i} \alpha_i r \E\big [ R^{(r)}_{e,i}   \mid Z^{(r)}_j=0\big ]
		+      \sum_{i} \mu^{(r)}_i r \E\big [ R^{(r)}_{s,i} \mid Z^{(r)}_j=0\big ] \Big).
		\end{align*}
		Now,
		\begin{align*}
		&  r \E\big [ R^{(r)}_{e,i} \mid Z^{(r)}_j=0\big ]
		=
		\frac{ r \E\big [  R^{(r)}_{e,i}  1_{\{ Z^{(r)}_j=0\}}\big ]}{\Prob\{Z^{(r)}_j =0\}}
		\le
		\frac{ r \big(\E\big [ \big(  R^{(r)}_{e,i}\big)^p\big]\big)^{1/p}  \big(\Prob\{ Z^{(r)}_j=0\}\big)^{1/q}}{\Prob\{Z^{(r)}_j =0\}} \\
		& = r^{1-j/p}   \Big(\E\big [ \big(  R^{(r)}_{e,i}\big)^{(J+\delta_0)}]\Big)^{1/(J+\delta_0)} (\mu_j^{(r)})^{1/p}=o(1) \quad \text{ for } i, j\in \mathcal{J},
		\end{align*}
		where in obtaining the  inequality, we have used H\"older's inequality with $p=J+\delta_0$ ($\delta_0$ is given in \eq{moment}) and $1/p+1/q=1$, and the last equality follows from Lemma 6.6 of \cite{BravDaiMiya2024} and Assumption \ref{ass:moment}. Similarly, one can prove
		\begin{align*}
		& r \E\big [ R^{(r)}_{s,i}  \mid Z^{(r)}_j=0\big ]=o(1)
		\quad \text{ for } i, j\in \mathcal{J}.
		\end{align*}
	\end{proof}

To prove Lemma~\ref{lem:ssc-phi}, we first establish Lemma~\ref{lem:moment-ssc-b}.
Both Lemma~\ref{lem:ssc-phi} and Lemma~\ref{lem:moment-ssc-b} are consequences of the uniform moment bounds in \pro{zbound}.  
\begin{lemma}\label{lem:moment-ssc-b}
	Assume the assumptions in \pro{zbound}. Then
	\begin{align}
	& \lim_{r\to 0} \E\Big[ \sum_{\ell<j} r^{\ell+1}Z^{(r)}_\ell  \mid Z_j^{(r)}=0 \Big] =0
	\quad \text{ for } j \in  {\cal{J}}. \label{eq:moment-ssc-b}
	\end{align}
\end{lemma}
\begin{proof}
	For each $j\in \mathcal{J}$ and $\ell< j$, 
	by choosing $p=j+\epsilon_0$, $q=(j+\epsilon_0)/(j+\epsilon_0-1)$ and employing the H\"older's inequality
	\begin{align*}
	\E\Big[ r^\ell Z^{(r)}_\ell 1_{\{ Z_j^{(r)}=0\}}\Big] \le
	\Big(\E\big[ (r^\ell Z^{(r)}_\ell)^p\big]\Big)^{1/p}\bigl (  \Prob\{ Z_j^{(r)}=0\}\bigr )^{1/q},
	\end{align*}
	one can show that
	\begin{align*}
	r \E\Big[ r^\ell Z^{(r)}_\ell \mid Z_j^{(r)}=0\Big] 
	& \le r    \Big(\E\big[ (r^\ell Z^{(r)}_\ell)^p\big]\Big)^{1/p}\bigl (  \Prob\{ Z_j^{(r)}=0\}\bigr )^{1/q-1}\\
	&=  \Big(\E\big[ (r^\ell Z^{(r)}_\ell)^p\big]\Big)^{1/p} (\mu^{(r)}_j)^{1/(j+\epsilon_0)}
	r^{\frac{\epsilon_0}{j+\epsilon_0}} \to 0,
	\end{align*}
	which proves  (\ref{eq:moment-ssc-b}).
\end{proof}       
 

\begin{proof}[Proof of Lemma~\ref{lem:ssc-phi}] In the proof, we drop the superscript $r$ from $\theta^{(r)}$.
  We first prove (\ref{eq:ssc1}). Note that $\theta_k\le 0$
          for $k< j$ and  $\theta_j-r^j\eta_j\le 0$. 
		\begin{align*}
		&\phi^{(r)}(0,\ldots, 0, r^j\eta_j, \ldots, r^J\eta_J)-\phi^{(r)}(\theta)
		= \E\Big[ \Big(1-e^{\sum_{k< j}\theta_k Z^{(r)}_k +(\theta_j-r^j\eta_j)Z^{(r)}_j}\Big)e^{\sum_{k\ge  j}\eta_k r^k Z^{(r)}_k} \Big] \\
		& \le  \E\Big[ 1-e^{\sum_{k< j}\theta_k Z^{(r)}_k +(\theta_j-r^j\eta_j)Z^{(r)}_j} \Big] \le
		\E\Big[ \sum_{k< j}\abs{\theta_k} Z^{(r)}_k +\abs{\theta_j-r^j\eta_j}Z^{(r)}_j  \Big] \\
		& = \sum_{k< j}\frac{\abs{\theta_k}}{r^{k}}     \E\Big[ r^{k} Z^{(r)}_k \Big]
		+\frac{\abs{\theta_j-r^j\eta_j}}{r^{j}} \E\Big[ r^j Z^{(r)}_j \Big] =o(1),
		\end{align*}
		where the second inequality follows from $1-e^{-x}\le x$ for $x\ge 0$ and the last equality is due to $\theta_k=o(r^{k})$ for $k<j$,
		$\theta_j-r^j\eta_j =o(r^{j})$, and \eq{Jmoment}.
		
	The proof of  (\ref{eq:ssc3}) follows from the following relationships
			\begin{align*}
				&\phi^{(r)}(0,\ldots, 0, r^{j+1}\eta_{j+1}, \ldots, r^J\eta_J)-\phi^{(r)}(\tilde{\theta})
				= \E\Big[ \Big(1-e^{\sum_{k\le j}\tilde{\theta}_k Z^{(r)}_k}\Big)e^{\sum_{k>  j}\eta_k r^k Z^{(r)}_k} \Big] \\
				& \le  \E\Big[ 1-e^{\sum_{k \le  j}\tilde{\theta}_k Z^{(r)}_k} \Big] \le
				\E\Big[ \sum_{k\le j}\abs{\tilde{\theta}_k} Z^{(r)}_k \Big] =
				\sum_{k\le j}\frac{\abs{\tilde{\theta}_k}}{r^{k}}     \E\Big[ r^{k} Z^{(r)}_k \Big] =o(1),
			\end{align*}
			where the last equality is due to $\tilde{\theta}_k=O(r^{j+1/2})=o(r^{k})$ for $k\le j$
			and \eq{Jmoment}.		
			
		We next prove  (\ref{eq:ssc2}).
		\begin{align*}
		&\phi^{(r)}_j(0,\ldots, 0, r^j\eta_j, \ldots, r^J\eta_J)-\phi^{(r)}_j(\theta)
		= \E\Big[ \Big(1-e^{\sum_{k< j}\theta_k Z^{(r)}_k}\Big)e^{\sum_{k\ge  j}\eta_k r^k Z^{(r)}_k}\mid Z^{(r)}_j=0 \Big] \\
		& \le  \E\Big[ \Big(1-e^{\sum_{k< j}\theta_k Z^{(r)}_k}\Big) \mid Z^{(r)}_j=0  \Big] \le
		\E\Big[ \sum_{k< j}\abs{\theta_k} Z^{(r)}_k \mid Z^{(r)}_j=0 \Big] \\
		& =
		\sum_{k< j}\frac{\abs{\theta_k}}{r^{k+1}}     \E\Big[ r^{k+1} Z^{(r)}_k \mid Z^{(r)}_j=0 \Big] =o(1),
		\end{align*}
		where the last equality is due to $\abs{\theta_k}\le
              c_k r^j$ for some constant $c_k>0$ as $r\to 0$ for $k<j$ and (\ref{eq:moment-ssc-b}).
              
The proof of  (\ref{eq:ssc4}) follows from the following relationships
	\begin{align*}
		&\phi^{(r)}_j(0,\ldots, 0, r^{j+1}\eta_{j+1}, \ldots, r^J\eta_J)-\phi^{(r)}_j(\tilde{\theta})
		= \E\Big[ \Big(1-e^{\sum_{k< j}\tilde{\theta}_k Z^{(r)}_k}\Big)e^{\sum_{k>  j}\eta_k r^k Z^{(r)}_k} \mid Z^{(r)}_j=0 \Big] \\
		& \le  \E\Big[ \Big(1-e^{\sum_{k< j}\tilde{\theta}_k Z^{(r)}_k} \Big) \mid Z^{(r)}_j=0 \Big] \le
		\E\Big[ \sum_{k< j}\abs{\tilde{\theta}_k} Z^{(r)}_k  \mid Z^{(r)}_j=0 \Big] \\
		& =
		\sum_{k< j}\frac{\abs{\tilde{\theta}_k}}{r^{k+1}}     \E\Big[ r^{k+1} Z^{(r)}_k  \mid Z^{(r)}_j=0\Big] =o(1),
	\end{align*}
	where the last equality is due to $\tilde{\theta}_k=O(r^{j+1/2})=o(r^{k+1})$ for $k< j$ and (\ref{eq:moment-ssc-b}).
	\end{proof}

        \section{Taylor expansions}
        \label{app:C}
        In this section, we prove Lemma~\ref{lem:taylor-psi}. The proof
        also covers the proof of Lemma~\ref{lem:taylor} by choosing
        $g^{(r)}(u)$ in \eq{gr} to be $u$, independent of $r$.
	\begin{proof}[Proof of Lemma~\ref{lem:taylor-psi}]
 For $r\in (0,1)$, define
          \begin{align}\label{eq:gr}
        g^{(r)}(u)=u \wedge r^{-1} \quad  \text{ for $u\in \R_+$}.    
          \end{align}
                    Fix $i\in \mathcal{J}$. 
	 Let
		\begin{align*}
		a^{(r)} = 1/\E\big[g^{(r)}\big(T_{e,i}(1)\big)\big], \quad        b^{(r)} = \E\big[ \big (g^{(r)}(T_{e,i}(1))\big)^2\big], \quad \sigma^2_{(r)} = b^{(r)}-(1/a^{(r)})^2.
		\end{align*}
		For $y\in \R$, define 
		\begin{align*} 
		c^{(r)}_{e,i}(y) = \gamma_i(y, r) -y a^{(r)} - \frac{1}{2} y^2 (a^{(r)})^3 \sigma^2_{(r)}.
		\end{align*}
		Lemma 4.2 of \cite{BravDaiMiya2017} proved that there exists $K>0$,
		\begin{align*}
		\lim_{r\to 0} \sup_{0<\abs{y}\le K}  \frac{c^{(r)}_{e,i}(ry)}{r^2y^2} = 0,
		\end{align*}
		which implies that for any $\theta_i(r)\in \R$ satisfying \eq{thetaor}, 
		\begin{align}
		\label{eq:ce}
		c^{(r)}_{e,i}(\theta_i(r)) = o((\theta_i(r))^2)  \quad \text{ as  }  r\to 0.
		\end{align}
		Note that 
		\begin{align*}
		& \E[T_{e,i}(1)] -       \E[g^{(r)}(T_{e,i}(1))] = \E[(T_{e,i}(1) -r^{-1}) \mathbbm{1}(T_{e,i}(1)> r^{-1})] \\
		& \le r^{J+\delta_0}  \E([T_{e,i}(1)^{J+1+\delta_0} \mathbbm{1}(T_{e,i}(1)> r^{-1})]   = o(r^{J+\delta_0})=o(r^J), \\
		& \E[(T_{e,i}(1))^2] -       \E[(g^{(r)}(T_{e,i}(1)))^2] = \E\big[(T^2_{e,i}(1)-r^{-2}) \mathbbm{1}(T_{e,i}(1)>r^{-1})\big]\\
		& \le r^{J-1+\delta_0}  \E([T_{e,i}(1)^{J+1+\delta_0} \mathbbm{1}(T_{e,i}(1)> r^{-1})]   = o(r^{J-1+\delta_0})=o(r^{J-1}).
		\end{align*}
	 where the first inequality is by applying $\frac{1}{r^{J+\delta_0}}\mathbbm{1}(T_{e,i}(1)> r^{-1})\leq \left(T_{e,i}(1)\right)^{J+\delta_0} \mathbbm{1}(T_{e,i}(1)> r^{-1})$, and similar technique applies on the second inequality. 
		Therefore, using the facts that $\E[T_{e,i}(1)]=1$ and $c^2_{e,i}={\rm Var}(T_{e,i}(1))$,
		\begin{align*}
		&  a^{(r)} = 1/  \E[g^{(r)}(T_{e,i}(1))] = 1/(1+o(r^{J})) = 1+o(r^{J}), \\
		& b^{(r)} =  \E[T_{e,i}(1)]^2 + o(r^{J-1}), \\
		& \sigma^2_{(r)} = c^2_{e,i} + o(r^{J-1}), \\
		& (a^{(r)})^3 \sigma^2_{(r)} = c^2_{e,i} + o(r^{J-1}).
		\end{align*}
		Finally, (\ref{eq:eta-t-taylor}) follows from (\ref{eq:ce}) and
		\begin{align*}
		& a^{(r)}\theta_i - \theta_i = \theta_io(r^{J})=o(r^{J}\theta_i),\\ 
		&  (a^{(r)})^3 \sigma^2_{(r)} \theta^2_i - c^2_{e,i}\theta^2_i = \theta^2_io(r^{J-1})=o(r^{J-1}\theta^2_i)=o(\theta_i^2).
		\end{align*}
		Expansion    (\ref{eq:xi-t-taylor}) can be proved similarly, again using Lemma~4.2 of \cite{BravDaiMiya2017}.
		
              \end{proof}        
	\addtocontents{toc}{\endgroup}
\end{appendix}	

\end{document}